\documentclass[10pt]{article}

\usepackage{amssymb, psfrag,graphicx,booktabs,color}
\usepackage{amsmath, amsthm}
\usepackage{algorithmicx}
\usepackage[ruled]{algorithm}
\usepackage{algpseudocode}

\newenvironment{alginc}[1][pseudocode]{\medskip
\algsetlanguage{#1}\begin{algorithmic}[1]}{\end{algorithmic}\medskip}

\newtheorem{theorem}{Theorem}
\newtheorem{proposition}[theorem]{Proposition}
\newtheorem{lemma}[theorem]{Lemma}

\newtheorem{remark}[theorem]{Remark}

\newcommand\di{\partial}
\newcommand\al{\alpha}

\newcommand\om{\omega}

\newcommand\ph{\varphi}

\newcommand\dm{d}
\newcommand\step{\rm samp}

\newcommand\R{\mathbb{R}}
\newcommand\C{\mathbb{C}}
\newcommand\Z{\mathbb{Z}}

\newcommand\Hs{\mathbb{H}}

\newcommand{\req}[1]{(\ref{eq:#1})}

\newcommand{\f}{\mathbf}
\newcommand{\fs}{\boldsymbol}

\DeclareMathOperator{\sign}{sign} 
\DeclareMathOperator{\sinc}{sinc} 

\DeclareMathOperator{\Qo}{\mathbf Q}
\DeclareMathOperator{\Mo}{\mathbf M}

\DeclareMathOperator{\Ft}{\mathbf F}

\newcommand\norm[1]{\|#1\|}
\newcommand\abs[1]{|#1|}
\newcommand\set[1]{\{#1\}}

\allowdisplaybreaks[1]

\begin{document}

\title{A Reconstruction Algorithm for Photoacoustic Imaging based on the Nonuniform FFT}

\author{Markus~Haltmeier,~Otmar~Scherzer,~and~Gerhard~Zangerl.
\thanks{M. Haltmeier, O. Scherzer, and  G. Zangerl are with the Department
of Mathematics, University Innsbruck, Technikerstr.~21a, 6020 Innsbruck, Austria,~
e-mail: \tt{\{markus.haltmeier, otmar.scherzer, gerhard.zangerl\}@uibk.ac.at}.}
\thanks{O.~Scherzer is also with the Radon Institute of Computational and Applied Mathematics,
Altenberger Str.~69, 4040 Linz, Austria}}

\markboth{header}{header 2}

\maketitle

\begin{abstract}
Fourier reconstruction algorithms  significantly
outperform conventional back-projection algorithms in terms of computation time.
In photoacoustic imaging, these  methods require  interpolation in the Fourier space domain,
which creates artifacts in reconstructed images.
We propose a novel reconstruction algorithm that applies the one-dimensional nonuniform
fast Fourier transform to photoacoustic imaging.
It is shown theoretically and numerically that our algorithm avoids artifacts while preserving the
computational effectiveness of Fourier reconstruction.

\medskip \noindent {\bf Key--words.} Image reconstruction, photoacoustic imaging, planar measurement geometry, fast algorithm,
nonuniform FFT.
\end{abstract}

\section{Introduction}
Photoacoustic imaging (PAI) is a novel promising tool for
visualizing light absorbing structures in an optically scattering medium,
which carry valuable information for medical diagnostics.
It is based on the generation of acoustic waves by illuminating an object with pulses of
non-ionizing  electromagnetic radiation, and combines the high contrast of pure optical and the high resolution of
ultrasonic imaging.  The method has demonstrated great promise for a variety of biomedical applications, such as
imaging of animals \cite{KruKisReiKruMil03,WanPanKuXieSto03}, early
cancer diagnostics \cite{KruMilReyKisReiKru00,ManKhaHesSteLee05},
and imaging of vasculature \cite{KolHonSteMul03,EseLarLarDeyMotPro02}.

When an object is illuminated with short pulses of non-ionizing electromagnetic radiation, it absorbs a fraction of energy and heats up. This in turn induces acoustic (pressure) waves, that are recorded with acoustic detectors outside of the object.
Other than in conventional ultrasound imaging, where the source of acoustic waves is an external transducer,
in PAI the source is the imaged object itself. The frequency bandwidth of the recorded  signals is therefore generally broad and depends on the size and the shape of  illuminated structures.

\subsection{Planar recording geometry}

Throughout this paper we assume  a planar  recording  geometry, where the acoustic signals are recorded with
omnidirectional detectors arranged on planes (or lines), see Figure~\ref{fg:setup0}.
The planar geometry is of particular interest since it can be realized most easily in practical applications.
The recorded acoustic signals are then used to reconstruct the initially generated acoustic pressure
which represents optically absorbing structures of the investigated object.

\psfrag{transducer array}{detector array} \psfrag{optical
illumination}{optical illumination} \psfrag{object}{object}
\psfrag{pressure}{acoustic wave}
\begin{figure}[htb]
\centering
\includegraphics[width=0.5\textwidth]{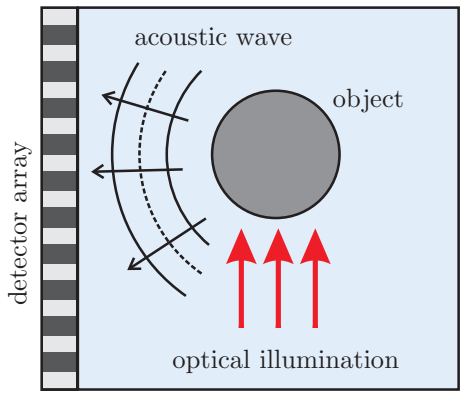}
\caption{{\bf Photoacoustic imaging for planar recording geometry.}
The object is illuminated by a pulse of electromagnetic radiation,
and reacts with an expansion. Induced acoustic waves are measured with an array of  acoustic
detectors arranged on a plane (or a line) and used to from an image of the object.} \label{fg:setup0}
\end{figure}

For the planar recording  geometry, two types of theoretically exact reconstruction formulas have been reported:
Temporal back-projection \cite{And88,BurBauGruHalPal07,Faw85,XuWan05} and Fourier domain formulas
\cite{AnaZhaModRiv07,And88,Faw85,NorLin81,XuFenWan02,KoeFraNiePalWeb01,KoeFreBebWeb01}.
Numerical implementations of those formulas often lead to fast and accurate image
reconstruction algorithms.

In temporal back-projection formulas the signals measured at time $t$ are back
projected over spheres of radius $v_s t$ with the detector position in the centre
($v_s$ denotes the speed of sound).
In Fourier domain formulas  this back-projection is performed by interpolation in the frequency
domain.  Reconstruction methods based on Fourier domain formulas are attractive since they
reconstruct an $N \times N \times N$ image in $\mathcal O (N^3 \log N)$ floating point operations
by using of the Fast Fourier Transform (FFT).
Straightforward implementations of back-projection type formulas, on the other hand,
require $\mathcal O (N^5)$ operation counts, see  \cite{HalSchuSch05,Nat86}.

The standard FFT algorithm assumes sampling on an equally spaced grid and therefore, in order
to implement the Fourier domain formulas, interpolation in the Fourier space
is required. Interpolation in the Fourier domain is a  critical issue, and creates artifacts
in reconstructed images, see the examples in Section \ref{sec:num}. One obtains significantly better results by
increasing the  sampling  density in the Fourier space. This is achieved by either zero-padding
\cite{XuFenWan02} or by symmetrizing the recorded signals around $t = 0$ (which is equivalent to
using the fast cosine transform instead of the FFT).
In this paper, we propose an efficient reconstruction algorithm that uses the nonuniform (or unevenly spaced)
FFT \cite{Bey95,DutRok93,FesSut03,GreLee04,PotSteTas01,Ste98} and further increases the quality of
reconstruction.

\subsection{Prior work and innovations}
The nonuniform FFT has been applied to a variety of medical imaging problems,
such as standard X-ray CT, magnetic resonance imaging, and diffraction
tomography \cite{BroBroZibAzh02,Fes07,GotGusFor01}, and has also been used implicitly
in gridding algorithms \cite{Osu85, SchoTim95}.
All those algorithms deal with the problem of recovering a two (or higher) dimensional object
function from samples of its multidimensional Fourier transform on a non-cartesian grid.

Our approach is conceptually  quite different to the above mentioned references:
The special structure of our problem allows to perform several one-dimensional nonuniform
FFTs instead of a single higher dimensional one.
This leads to a reduced numerical cost,  compared to the above algorithms.
The proposed algorithm   is  more  closely related to a reconstruction algorithm
for X-ray CT suggested in \cite[Section 5.2]{Fou03}, which also evaluates the Fourier
transform on  irregular samples by means of the one-dimensional FFT.

\subsection{Outline}

\medskip
This article is organized as follows: In Section \ref{sec:problem}
we present the mathematical basics of Fourier reconstruction in
PAI. In Section \ref{sec:nufft} we review the
nonuniform FFT which is then used to derive the nonuniform FFT based
reconstruction algorithm in Section \ref{sec:fra}. In Section
\ref{sec:num} we present numerical results of the proposed algorithm and compare it
with existing Fourier and back projection algorithms.
The paper concludes with a discussion of some
issues related to sampling and resolution in the Appendix.

\section{Photoacoustic Imaging} \label{sec:problem}

Let $C_0^\infty(\Hs)$  denote the space of smooth functions
with bounded support in the half space  $\Hs := \R^{d-1} \times (0, \infty)$,
$d \geq 2$.
Consider the initial value problem
\begin{align*}
    \left(\di_t^2 - \Delta \right) p(\f x, t) & = 0 \,,
    &(\f x,t) \in \R^{\dm} \times (0, \infty)\,,
    \\
    p(\f x, 0)  &=  f(\f x)
    \,, &
    \f x \in \R^{\dm}  \,,
    \\
    \di_t p( \f x, 0) & =  0\,,
    &\f x \in \R^{\dm} \,,
\end{align*}
with $f\in C_0^\infty(\Hs)$.  Here $\Delta$ denotes the Laplacian
in $\R^{\dm}$ and $\di_t$ is the derivative with respect to $t$. We
write $\f x=(x, y)$, $x \in \R^{d-1}$, $y \in \R$, and define the
operator $\Qo: C_0^\infty(\Hs) \to C^\infty(\R^d)$ by
\begin{equation*}
    (\Qo f) (x, t) := \left\{
                        \begin{array}{ll}
                          p( x, y=0, t ) \,, & \hbox{ if } t > 0\,, \\
                          0 \,, & \hbox{ otherwise} \,.
                        \end{array}
                      \right.
\end{equation*}
Photoacoustic imaging for planar recording geometry is concerned with reconstructing $f \in
C_0^\infty(\Hs)$ from incomplete and possibly erroneous knowledge of
$\Qo f$.  Of practical interest are the cases $d=2$ and $d=3$,
see \cite{KucKun08,PatSch07,SchGraGroHalLen09,XuWan06}.

\subsection{Exact inversion formula}

The operator  $\Qo$ can be inverted analytically by means of the
exact inversion formula
\begin{equation} \label{eq:inv}
     (\Ft f) (K_x, K_y)
    =
    \frac{2 K_y   \bigl( \Ft \Qo f \bigr)
    \left( K_x, \sign(K_y)
    \sqrt{\abs{K_x}^2 + K_y^2}\right) }{\sign(K_y) \sqrt{\abs{K_x}^2 + K_y^2 } }
\end{equation}
where   $(K_x, K_y)  \in \R^{d-1} \times  \R$, and $\Ft $ denotes the $d$-dimensional
Fourier  transform,
\begin{equation*}
    (\Ft \ph) (\f K)
    :=
    \int_{\R^{d}} e^{- i \f K \f x} \ph(\f x) \ d\f x \,,
    \qquad \f K = (K_x, K_y) \in \R^d \,.
\end{equation*}
Equation \req{inv} has been derived in \cite{NorLin81, XuFenWan02}
for three spatial dimensions. It can be proven in any
dimension by using the inversion formula for the spherical mean Radon
transform of \cite{And88, Faw85}. A related formula using
the Fourier cosine transform instead of the Fourier transform has
been obtained in \cite{KoeBea03, KoeFreBebWeb01} for $d = 2,3$.

\subsection{Partial Fourier reconstruction}

The inversion formula \req{inv} yields an exact reconstruction of $f$,
provided that $(\Qo f) (x, t )$ is given for all $(x, t)\in \R^d$.
In practical applications only a \emph{partial} (or \emph{limited
view}) data set is available \cite{PalNusHalBur07b, PanAna02, XuYWanAmbKuc04}.
In this paper we assume that data $( \Qo f ) (x, t )$ are  given only for $ (x,t) \in (0,X)^d$, see
Figure \ref{fg:setup0}, which are modeled by
\begin{equation}\label{eq:data-part}
    g(x, t) := w_{\rm cut}(x,t)
                   ( \Qo f ) (x, t ) \,,
\end{equation}
where $w_{\rm cut}$ is a smooth nonnegative cutoff function  that vanishes outside
$(0,X)^d$.  Using  data \req{data-part}, the function $f$ cannot be exactly
reconstructed in a stable way (see \cite{LouQui00, XuYWanAmbKuc04}).
It is therefore common to apply the exact inverse of $\Qo$ to the
partial data $g$ and to consider the result as an approximation of the
object to be reconstructed. More precisely, the function $f^\dag$ defined
by
\begin{equation}\label{eq:inv2}
    (\Ft f^\dag) (K_x, K_y)
    :=
    \frac{2 K_y  ( \Ft g )
    \left( K_x,\sign(K_y) \sqrt{\abs{K_x}^2 + K_y^2}\right) }
    {\sign(K_y)\sqrt{\abs{K_x}^2 + K_y^2 } }\,,
\end{equation}
is considered an approximation of $f$.
The function $f^\dag$ is called \emph{partial Fourier reconstruction}.

Fourier reconstruction algorithms in PAI name numerical implementations of \req{inv2}.
In this paper we apply the one-dimensional nonuniform FFT to derive a fast and accurate
algorithm for implementing \req{inv2}.

\section{The Nonuniform Fast Fourier Transform}
\label{sec:nufft}

The discrete Fourier transform of a vector $\f g = (g_n)_{n
=0}^{N-1} \in \C^N$ with respect to the nodes $\fs \omega =
(\omega_k)_{k = -N/2}^{N/2-1}$  (with $N$ even) is defined by
\begin{equation} \label{eq:dft}
    T[\f g](\omega_k)
    :=
    \sum_{n=0}^{N-1}
    e^{-i  \omega_k n 2\pi/N  }
    g_n
    \,, \quad k = -N/2, \dots, N/2-1 \,.
\end{equation}
Direct evaluation of the $N$ sums in \req{dft} requires $\mathcal
O(N^2)$ operations. Using the classical fast Fourier transform
(FFT) this effort can be reduced to $\mathcal
O(N \log N)$ operations. However, application of the classical FFT
is restricted to the case of evenly spaced nodes $\omega_k =k$, $ k=
-N/2, \dots, N/2-1$.

The one-dimensional nonuniform FFT (see \cite{Bey95,DutRok93,FesSut03,Fou03,GreLee04,PotSteTas01,Ste98})
is an approximate but highly accurate method for evaluating \req{dft} at arbitrary
nodes $\omega_k$, $ k= -N/2, \dots, N/2-1$, in $\mathcal O(N \log N)$
operations.

\subsection{Derivation of the nonuniform FFT}

To derive the nonuniform FFT we closely follow the presentation of \cite{Fou03},
which is based on the following lemma:

\begin{lemma}[{\cite[Proposition 1]{Fou03}}]\label{le:fou}
Let $c > 1$ and $\al < \pi(2c-1)$. Assume that $\Psi:\R \to \R$ is
continuous in $[-\al, \al]$, vanishing outside $[-\al, \al]$, and
positive in $[-\pi, \pi]$. Then
\begin{equation}\label{eq:fourmont}
    e^{-i \omega \theta}
    =
    \frac{c}{2\pi \Psi(\theta)}
    \sum_{j \in \Z}
    \hat \Psi (\omega - j/c)
    e^{-i  j \theta /c }
    \,, \quad
    \omega \in \R
    \,,
    \abs{\theta} \leq  \pi \,.
\end{equation}
Here $\hat \Psi(\omega) := \int_{\R} e^{-i\omega\theta} \Psi(\theta)
d\theta$ denotes the one-dimensional Fourier transform of $\Psi$.
\end{lemma}

\begin{proposition} \label{thm:nufft}
Let $c$, $\alpha$, $\Psi$, and $\hat \Psi$ be as in Lemma
\ref{le:fou}. Then, for every $\f g = (g_n)_{n=0}^{N-1} \in \C^N$
and $\omega \in \R$ we have
\begin{equation} \label{eq:nufft-a}
   \sum_{n=0}^{N-1}
     e^{-i  \omega n 2\pi/N}
    g_n
   =
   \sum_{j \in \Z}
    e^{- i \pi (\omega - j/c)}
    \hat \Psi (\omega - j/c)
    \hat G_j
    \,,
\end{equation}
with
\begin{equation} \label{eq:nufft-b}
\hat G_j
   :=
   \frac{c}{2\pi}
   \left(
   \sum_{n =0}^{N-1}
    \frac{g_n e^{-i j n 2\pi/(Nc)}}{ \Psi(n 2\pi/N - \pi)}
\right)
    \,, \qquad j \in \Z \,.
\end{equation}
\end{proposition}

\begin{proof}
Taking  $\theta = n 2\pi/N - \pi  \in [-\pi,\pi]$ in \req{fourmont},
gives
\begin{equation*}
    e^{-i \omega n  2\pi/N}
    =
    \frac{c}{2\pi \Psi(n 2\pi/N - \pi)}
   \sum_{j \in \Z}
    \hat \Psi(\omega - j/c)
    e^{-i j n 2\pi/(cN)}
    e^{-i \pi (\omega-j/c)} \,,
\end{equation*}
and therefore
\begin{equation*}
  \sum_{n =0}^{N-1}
  e^{-i \omega n  2\pi/N} g_n
   =
   \frac{c}{2\pi}
   \sum_{n =0}^{N-1}
   \sum_{j \in \Z}
     e^{-i \pi (\omega-j/c)}
     \hat \Psi(\omega - j/c)
    \frac{g_n e^{-i j n 2\pi/(cN)}}{ \Psi(n 2\pi/N - \pi)} \,.
\end{equation*}
Interchanging the order of summation in the right hand side of the
above equation shows \req{nufft-a}, \req{nufft-b}  and concludes the
proof.
\end{proof}

In the following we assume that $cN$ is an even number. Then
\begin{equation} \label{eq:nufft-c}
\hat G_j
   =
   \frac{c}{2\pi}
   \left(
   \sum_{n =0}^{c N-1}
    \frac{g_n }{ \Psi(n 2\pi/N - \pi)}  e^{-i j n 2\pi/(Nc)}
\right)
    \,, \quad j \in \Z \,,
\end{equation}
where  $g_n:=0$ for $n \geq N$, is an oversampled discrete
Fourier transform with  the oversampling factor $c$.
Moreover we assume that $\hat \Psi$ is concentrated
around zero and decays rapidly away from zero.
The nonuniform FFT uses the formulas \req{nufft-a}, \req{nufft-c}
to evaluate $T[\f g]$ at the nodes $\omega_k$.
The basic steps of the algorithm are as follows:

\begin{enumerate}
  \item[(i)]
  Append $(c-1)N$ zeros to the vector $\f g = (g_n)_{n=0}^{N-1}$ and evaluate
  $\hat G_j$, $j= -Nc/2, \dots,Nc/2-1$, in \req{nufft-c}
  with the FFT algorithm.

  \item[(ii)]
  Evaluate the sums in \req{nufft-a} approximately
  by using only the terms with $\abs{\omega_k -j/c} \leq   K$,
  where the \emph{interpolation length} $K$ is a small positive
  parameter.
\end{enumerate}
Since $\hat \Psi$ is assumed to decay rapidly, the truncation error in
Step~(ii) is small.

\begin{algorithm}[h]
\caption{Nonuniform FFT with respect to the nodes $\fs \omega =
(\omega_k)_{k = -N/2}^{N/2-1}$, using input vector $\f g =
(g_n)_{n=0}^{N-1}$, oversampling $c>1$, interpolation length $K$,
and window function $\Psi$.} \label{alg:nufft}

\begin{alginc}

\State
$\fs \Psi  \gets \bigl(\Psi (2\pi n / N - \pi)  \bigr)_n$
 \Comment{precomputations}

\State $\hat{\fs\Psi}  \gets \bigl( e^{-i (\omega_k- j/c)\pi/c} \hat
\Psi(\omega_k - j/c) \bigr)_{k,j}$

\State \Function{\tt nufft}{$\f g, \fs \omega, c, K, \fs \Psi, \hat{\fs \Psi}$}

\State $\f g \gets \f g / \fs \Psi \cdot c/(2\pi)$

 \State $\f g \gets \bigl( \f g, {\tt zeros}( 1 , (c-1)N \bigr)$
\Comment{zero-padding}

 \State $\f g \gets {\tt fft}(\f g)$
\Comment{one-dimensional FFT}

\For{$k=-N/2, \dots, N/2-1$}
\State
$\hat g_k
\gets
   \sum_{ \abs{j-c\omega_k} \leq cK }
   \hat \Psi_{k,j} g_j
$  \Comment{interpolation}
\EndFor

\State
\Return{$(\hat g_k)_k$}

\EndFunction
\end{alginc}
\end{algorithm}

The nonuniform Fourier transform is summarized in
Algorithm \ref{alg:nufft}. All evaluations of $\Psi$ and $\hat \Psi$ are
precomputed and stored. Moreover, the classical FFT is applied to a
vector of length $cN$. Therefore the numerical complexity of Algorithm \ref{alg:nufft}
is  $\mathcal O(c N \log N)$.
Typically $c = 2$, in which case the numerical effort of the nonuniform FFT is essentially
twice the effort of the one-dimensional classical FFT applied to an
input vector of the same length. See \cite[Section~3]{Fou03}
for an exact operation count, and a comparison between actual computation times of the
classical and the nonuniform FFT.

\subsection{Kaiser Bessel window}

In our implementation we choose for $\Psi$ the \emph{Kaiser
Bessel window},
\begin{equation*}
    \Psi^{\al, K}_{\rm KB}( \theta )
    :=
    \frac{1}{I_0(\al K)}
    \left\{
      \begin{array}{ll}
        I_0 (K \sqrt{\al^2 - \theta^2})
    \,, & \text{ if } \abs{\theta} \leq \al \,, \\
    0
    \,, & \text{ if } \abs{\theta} > \al  \,.
      \end{array}
    \right.
\end{equation*}
Here $I_0$ is the modified Bessel function of order zero. The one-dimensional Fourier transform of $\Psi^{\al, K}_{\rm KB}$ is
\begin{equation*}
    \hat \Psi^{\al, K}_{\rm KB}(\om )
    =
    2 \sinh (\al\sqrt{K^2 - \om^2})/\bigl( I_0(\al K) \sqrt{K^2-\om^2}\bigr)\,,
\end{equation*}
if $\omega  \in \R \setminus \set{-K,K}$ and $ 2 \alpha/(I_0(\al K))$ otherwise.

The Kaiser Bessel window is a good and often used candidate for
$\Psi$, since $\hat \Psi^{\al, K}_{\rm KB}(\omega)$ becomes
extremely small for $\abs{\omega} \geq K$. For example, with the
parameters $K = 3$, and $\al=3\pi$, we have for $\omega  \geq  K $,
\[
    \abs{ \hat\Psi^{\al, K}_{\rm KB} (\omega) / \hat\Psi^{\al, K}_{\rm KB} (0) } \leq
    \abs{ \hat\Psi^{\al, K}_{\rm KB} (K) / \hat\Psi^{\al, K}_{\rm KB} (0) }
    \simeq 3* 10^{-11} \,.
\]

\begin{figure}[htb]
\centering \includegraphics[width=0.45\textwidth]{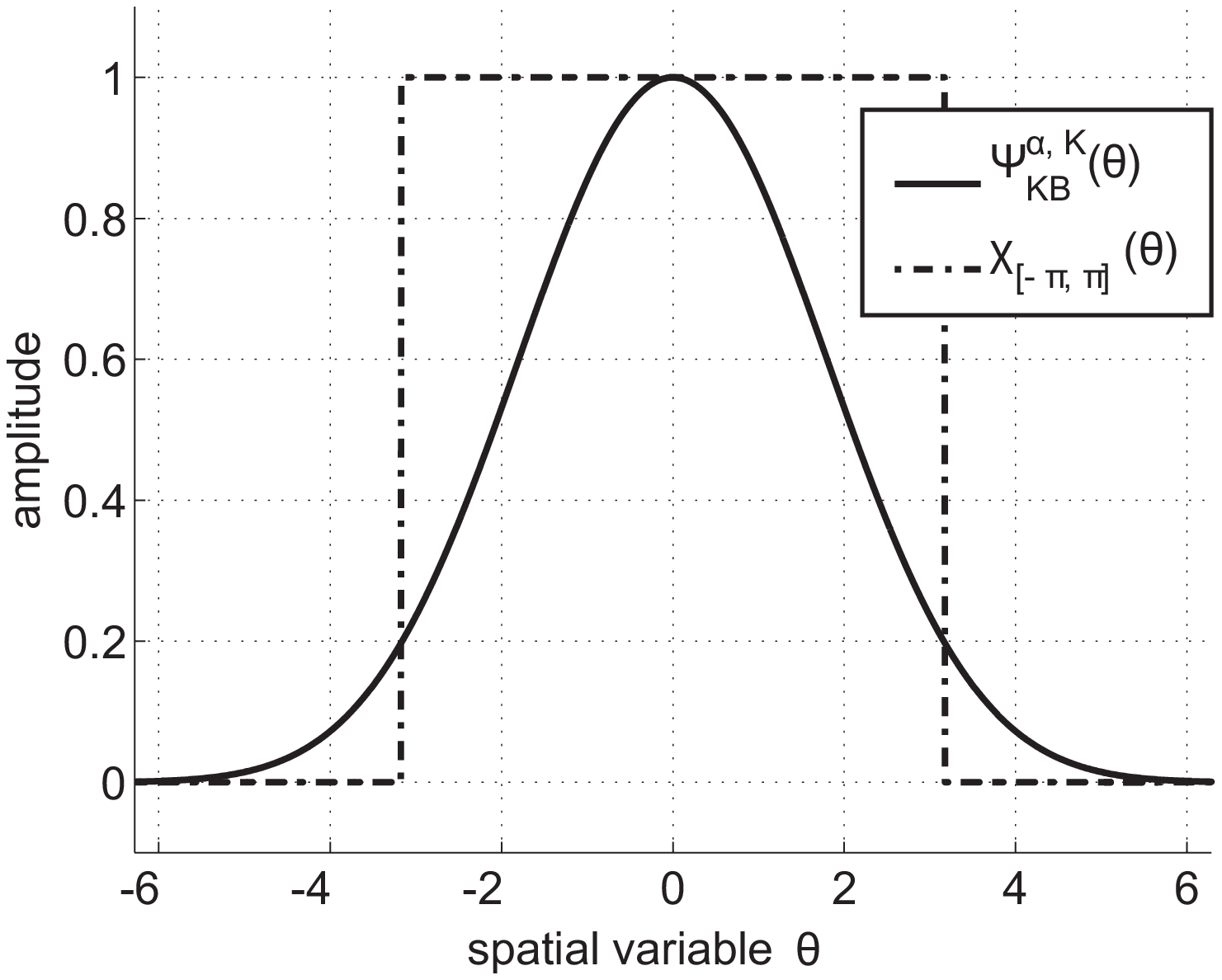}\quad 
\includegraphics[width=0.45\textwidth]{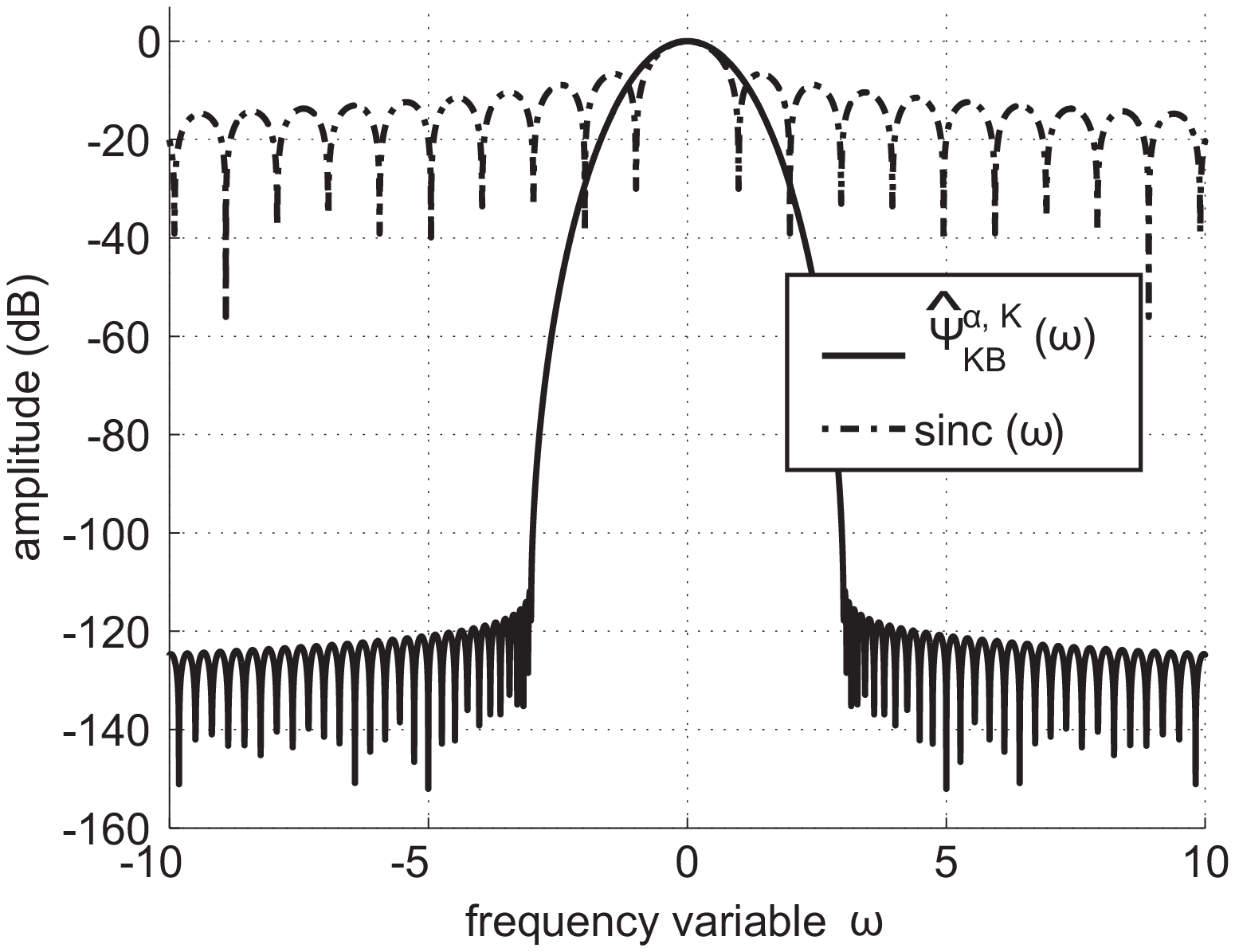}
\caption{
\emph{Left:} Kaiser-Bessel window $\Psi^{\al, K}_{\rm KB}(\theta)$ and characteristic function of
the interval $[-\pi, \pi]$.
\emph{Right:} Fourier transforms  $\hat \Psi^{\al, K}_{\rm KB}(\omega) $ and $2\pi\sinc(\omega)$
in dB (decibel). Here dB denotes the logarithmic decay  $10 \log_{10}( \abs{ \phi(\omega) / \phi(0) })$ of some quantity
$\phi(\omega)$.} \label{fg:decay}
\end{figure}

\begin{remark}\label{rem:nufft}
Take $c=1$ and let $\Psi$ be the characteristic function of the
interval $[-\pi,\pi]$. Then $\hat \Psi ( \om ) = 2\pi \sinc(\pi
\om)$ and \req{nufft-a}, \req{nufft-b}  reduce to the $\sinc$ series
\begin{equation*}
   \sum_{n=0}^{N-1}
   e^{-i  \omega n 2\pi/N}
   g_n
   =
   \sum_{j \in \Z}
    e^{- i \pi (\omega - j)}
    \sinc(\omega - j)
    \left(
    \sum_{n =0}^{N-1}
    g_n e^{-i j n 2\pi/N}
    \right)
    \,,
\end{equation*}
which is a discretized version of Shannon's sampling formula \cite{NatWue01,Uns00}
\begin{equation*}
   \hat g (\omega)
   =
   \sum_{j \in \Z}
    e^{- i \pi (\omega - j)}
    \sinc(\omega - j) \hat g (j)\,, \qquad \omega \in \R\,,
\end{equation*}
applied to the Fourier transform of a function $g:\R \to \R$ that vanishes outside
$[0,2\pi]$.

See Figure \ref{fg:decay} for a comparison
of $\sinc$ and $\hat \Psi^{\al, K}_{\rm KB}$, with $K=3$ and $\al =
3\pi$. One realizes that $\hat \Psi^{\al, K}_{\rm KB}$ decays much
faster than $\sinc$ and is therefore much better suited for truncated interpolation.
In fact, $\abs{ \hat\Psi^{\al, K}_{\rm KB} (\omega) / \hat\Psi^{\al, K}_{\rm KB} (0) } < 3* 10^{-11}$
for $\omega \geq 3$, whereas $\abs{\sinc (\omega)} < 0.01$ only for $\omega \geq 100/\pi$.
\end{remark}

An error estimate for the nonuniform FFT using the Kaiser Bessel
window
 is given in \cite{Fou03}. The result is
 \begin{multline*}
\Bigl|  e^{-i \omega \theta}
    -
    \frac{c}{2\pi \Psi(\theta)}
    \sum_{\abs{\omega - j/c} < K}
    \hat \Psi_{\rm KB}^{\alpha, K} (\omega - j/c)
    e^{-i  j \theta /c }
\Bigr|
\\
\leq \frac{30}{\pi I_0\bigl( K \pi  \sqrt{\alpha^2-1} /c^2 \bigr)}  \,.
\end{multline*}
For example, taking $c=2$, $\al = 3\pi$ and $K=3$, the above error
is as small as $3*10^{-8}$.

\section{A Fourier reconstruction Algorithm based on the
nonuniform FFT} \label{sec:fra}

In this section we apply the one-dimensional  nonuniform FFT to photoacoustic
imaging. Throughout the following we restrict our attention to
two dimensions, noting that the general case $d \geq 2$ can be
treated in an analogous manner.

Assume that $f$ is a smooth function that vanishes outside $(0,X)^2$,
and set $g := w_{\rm cut} \Qo f $, where $w_{\rm cut}$ is as in \req{data-part}.
Fourier reconstruction names an implementation of \req{inv2}, that uses
discrete data
\begin{equation} \label{eq:gd}
    g_{m,n}
    : =
    g (  m \Delta_{\step}, n \Delta_{\step}) \,,
\end{equation}
with  $(m,n) \in \set{0, \dots, N-1}^2$ and reconstructs an
approximation
\begin{equation} \label{eq:fd}
    f_{m,n}
    \simeq
    f^\dag (  m \Delta_{\step}, n \Delta_{\step}) \,,
\end{equation}
with $ (m,n) \in \set{0, \dots, N-1}^2$.
Here $f^\dag$ is defined by \req{inv2}, $N$ is an even number and
$\Delta_{\step} := X/N$. In the Appendix  we show
that the sampling in \req{fd}, \req{gd} is sufficiently fine,
provided that
$\Delta_{\step} \leq \pi/\Omega$, where $\Omega$ is the essential
bandwidth of $f$. 

Discretizing \req{inv2} with the trapezoidal rule  gives
\begin{multline} \label{eq:dft-2d}
    \sum_{n=0}^{N-1}
    \left( \sum_{m = 0}^{N-1}
    e^{-i  (l n+km)2\pi/N}
    f_{m,n}\right)
    \\ =
    \frac{2 k}{ \om_{k,l }}    \sum_{n=0}^{N-1}
    e^{-i  \omega_{k,l} n 2\pi/N}
    \left( \sum_{m = 0}^{N-1}
    e^{-i  k m 2\pi/N}
    g_{m,n}\right)\,,
\end{multline}
where
\begin{equation*}
    \omega_{k,l}
    : =  \sign( l ) \sqrt{k^2 + l^2} \,,
    \quad (k,l) \in \set{-N/2, \dots, N/2 -1}^2 \,.
\end{equation*}
One notices that the inner sums in \req{dft-2d},
\begin{equation} \label{eq:fft-1a}
    \tilde g_{k,n}
    :=
    \sum_{m = 0}^{N-1}
    e^{-i k m 2\pi/N}
    g_{m,n}
\end{equation}
can be exactly  evaluated with $N$ one-dimensional FFTs, and the outer sums
\begin{equation} \label{eq:fft-2}
    \hat g_k(\omega_{k,l})
    :=
    \sum_{n=0}^{N-1}
    e^{-i  \omega_{k,l} n 2\pi/N}
    \tilde g_{k,n}
\end{equation}
can be approximately evaluated with $N$ one-dimensional nonuniform FFTs.
Denoting the resulting approximation by $\hat g_{k,l} \simeq \hat g_k(\omega_{k,l})$
and setting
\begin{equation} \label{eq:fft-0}
\hat f_{k, l}  := \frac{2 k  \, \hat g_{k,l}}{ \om_{k,l}}
\,, \qquad
(k,l) \in \set{-N/2, \dots, N/2-1}^2\,,
\end{equation}
we finally find
\begin{equation} \label{eq:fft-1}
    f_{n,m}
    :=
    \frac{1}{N^2}
    \sum_{k,l = N/2}^{N/2-1}
    e^{i  (km + ln)  2\pi/N}
    \hat f_{k, l}
\end{equation}
with the inverse two-dimensional FFT algorithm.

\begin{algorithm} \caption{Nonuniform FFT based algorithm
for calculating $\f f = (f_{m,n})_{n,m=0}^{N-1}$ using data $\f g =
(g_{m,n})_{m,n=0}^{N-1}$, oversampling factor $c$, interpolation
length $K$, and window function $\Psi$.}
\label{alg:nff}

\begin{alginc}
\State
$\fs \Psi  \gets \bigl(\Psi (2\pi n / N - \pi)  \bigr)_n$
 \Comment{precomputations}

\State $\hat{\fs\Psi}  \gets \bigl( e^{-i (\omega_k- j/c)\pi/c} \hat
\Psi(\omega_k - j/c) \bigr)_{k,j}$

\State

\Function{\tt FouRecNufft}{$\f g, c, K, \fs \Psi, \hat{\fs \Psi}$}
\For{$n=0, \dots, N-1$} \State $\f h \gets  (g_{m,n})_{m}$ \State
$(\tilde g_{k,n})_{k}
        \gets
        {\tt fft}(\f h)
    $
 \Comment{one-dimensional FFT}

\EndFor

\State $\f l \gets (-N/2, \dots, N/2-1)$

\For{$k=-N/2, \dots, N/2-1$} \State $\fs \omega \gets {\tt sign}(\f
l) \sqrt{k^2+\f l^2} $ \State
    $\f h
     \gets {\tt nufft}(\f h, \fs \omega, c, K, \fs \Psi, \hat{\fs \Psi})$
    \Comment{nonuniform FFT}

\State $(f_{k,l})_l
    \gets
    2k \,  \f h / \fs \om
$ \EndFor

\State $\f f \gets (f_{k,l})_{k,l}$ \State $\f f \gets {\tt
ifft2}(\f f)$ \Comment{two-dimensional inverse FFT}

\State \Return{$\f f$}

\EndFunction
\end{alginc}
\end{algorithm}

The nonuniform FFT based reconstruction algorithm is summarized in
Algorithm \ref{alg:nff}. Its numerical complexity can easily be
estimated. Evaluating \req{fft-1a} requires $N \mathcal O(N \log N)$
operations ($N$  one-dimensional FFTs), evaluating \req{fft-2}
requires $N \mathcal O(N \log N)$ operations ($N$ non-uniform FFTs),
and \req{fft-1} is evaluated with the inverse two-dimensional FFT in
$\mathcal O(N^2 \log N)$ operations. Therefore the overall
complexity of Algorithm \ref{alg:nufft} is $\mathcal O(N^2 \log N)$.

In the next section we numerically compare Algorithm \ref{alg:nff} with standard Fourier algorithms presented
in the literature \cite{JaeSchuGerKitFre07,XuFenWan02},
which all differ in the way how the sums in \req{fft-2} are evaluated:

\begin{enumerate}
\item
\textbf{Direct Fourier algorithm.}
Equation \req{fft-2} cannot be  evaluated with the classical FFT algorithm
because the nodes $\omega_{k,l}$ are non-equispaced.
The most simple way to evaluate \req{fft-2} is with direct summation.
Because there are $N^2$ such sums, direct Fourier reconstruction requires
$\mathcal O(N^3)$ operations.
Consequently it does not lead to a fast algorithm.
However, since \req{fft-2} is evaluated exactly, it is optimally
suited to evaluate the image quality of  reconstructions  with  fast
Fourier algorithms.

\item
\textbf{Interpolation based algorithm.} A fast and
simple alternative to direct Fourier reconstruction is as
follows: Choose an oversampling factor $c\geq1$ and exactly  evaluate
\begin{equation*}
    \hat g_k(\omega)
    :=
    \Delta_{\step}
    \sum_{n=0}^{N-1}
    e^{-i  \omega n 2\pi / N }
    \tilde g_{k,n}\,,
\end{equation*}
at the uniformly spaced nodes $\omega = \Delta_{\step} j / c $,
$j\in \set{0, \dots, Nc-1}$, with the one-dimensional FFT algorithm.
In a next step, linear interpolation is used to find approximate values $\hat g_{k,l}
\simeq \hat g_k(\omega_{k,l})$, see \cite{XuFenWan02}.
Evaluating $\hat g_{k,l}$ with linear interpolation
requires $\mathcal{O} (N^2)$ operation and therefore the overall numerical effort
of linear interpolation based Fourier  reconstruction is $\mathcal O(N^2 \log N)$.

Algorithms using nearest neighbor interpolation instead of the linear one have the same
numerical complexity  and have also been applied to PAI (see, e.g., \cite{CoxArrBea07}).
Higher order polynomial interpolation has been
applied in \cite{Kun07b} for a cubic recording geometry.
\item
\textbf{Truncated $\sinc$ reconstruction.}
If the function
$\Psi$ in Algorithm \ref{alg:nff} is chosen as the
characteristic function of the interval $[-c\pi, c\pi]$, $c \geq
1$, then the nonuniform fast Fourier transform reduces to the
truncated $\sinc$ interpolation considered in \cite{JaeSchuGerKitFre07}.
However, due to the slow decay of $\sinc(\omega)$, truncation
will introduce a non-negligible error in the reconstructed image
(see Remark \ref{rem:nufft}).
\end{enumerate}

The Fourier algorithms are  also be compared with a numerical implementation of the
back-projection formula
\begin{equation}\label{eq:ubp}
        f(x,y)
        =
        - \frac{2 y}{\pi}
        \int_{\R}
        \left(
        \int_{r}^\infty
        \frac{ (\partial_t t^{-1} \Qo f)(x', t)}{ \sqrt{t^2-r^2}} dt
        \right)
        dx' \,,
\end{equation}
where $r = \sqrt{(x-x')^2 + y^2}$ denotes the distance between the detector location $(x', 0)$ and the
reconstruction point $(x, y)$.
Equation \req{ubp} has been obtained in  \cite{BurBauGruHalPal07}  by applying the method of descent to the three-dimensional
universal  back-projection formula discovered  by Xu and Wang \cite{XuWan05}.
Again \req{ubp} gives an exact reconstruction only if it is applied to complete data  $(\Qo f)(x,t)$, $(x,t)\in \R^2$.
In the numerical experiments the back-projection formula is applied to the partial data  $w_{\rm cut} \Qo f $,
see \req{data-part}, and implemented with $\mathcal O(N^3)$ operation  counts as described in  \cite[Section 3.3]{BurBauGruHalPal07}.

\section{Numerical Results}
\label{sec:num}

In the following we numerically compare the proposed nonuniform FFT based algorithm with
standard Fourier algorithm and  the back projection algorithm based on \req{ubp}.

The cutoff function $w_{\rm cut}$ is constructed by convolution of
\begin{equation*}
\ph_\epsilon(x,t) =
\begin{cases}
C_\epsilon \exp\bigl(-1/{(\epsilon-x^2-t^2)}^4\bigr),& \text{if } x^2 + t^2 < \epsilon\,,\\
0,&\text{ otherwise } \,,
\end{cases}
\end{equation*}
with the characteristic function of $[0,X]^2$, where $\epsilon$ is a
small parameter and $C_\epsilon$ is chosen in such a way that
$\int_{\R^2} \ph_\epsilon(x,t)dxdt = 1$.
Typically, $\epsilon$ is chosen as a ``small'' multiple of the sampling step size
$\Delta_{\rm samp} = X/N$.

In all numerical experiments we take $X=1$, and $N = 512$.
The window width $\alpha$ is chosen  to be  slightly smaller than $\pi(2c - 1)$,
where $c$ is the oversampling factor that determines the accuracy of the Fourier reconstruction
algorithms.

\psfrag{original}{\color{white} $f_{\rm circ}$}
\psfrag{datacirc}{\color{white} $\Qo f_{\rm circ}$}
\psfrag{originalphant}{\color{white} $f_{\rm phant}$}
\psfrag{dataphant}{\color{white} $\Qo f_{\rm phant}$}
\psfrag{BP}{\color{white} back projection}
\psfrag{back projection}{\color{white} back projection}
\psfrag{nearest, c=1}{\color{white} nearest, $c=1$}
\psfrag{nearest, c=2}{\color{white} nearest, $c=2$}
\psfrag{linear}{\color{white} linear, $c=2$}
\psfrag{, c=2}{}
\psfrag{lin}{\color{white} linear, $c=1$}
\psfrag{, c=1}{}
\psfrag{tsinc, c=2}{\color{white} truncated $\sinc$, $c=2$}
\psfrag{direct}{\color{white} direct}
\psfrag{kb, c = 2}{\color{white} nonuniform FFT, $c=2$}

\begin{figure}[t!]
\centering
\includegraphics[width=0.325\textwidth,height=0.3\textwidth]{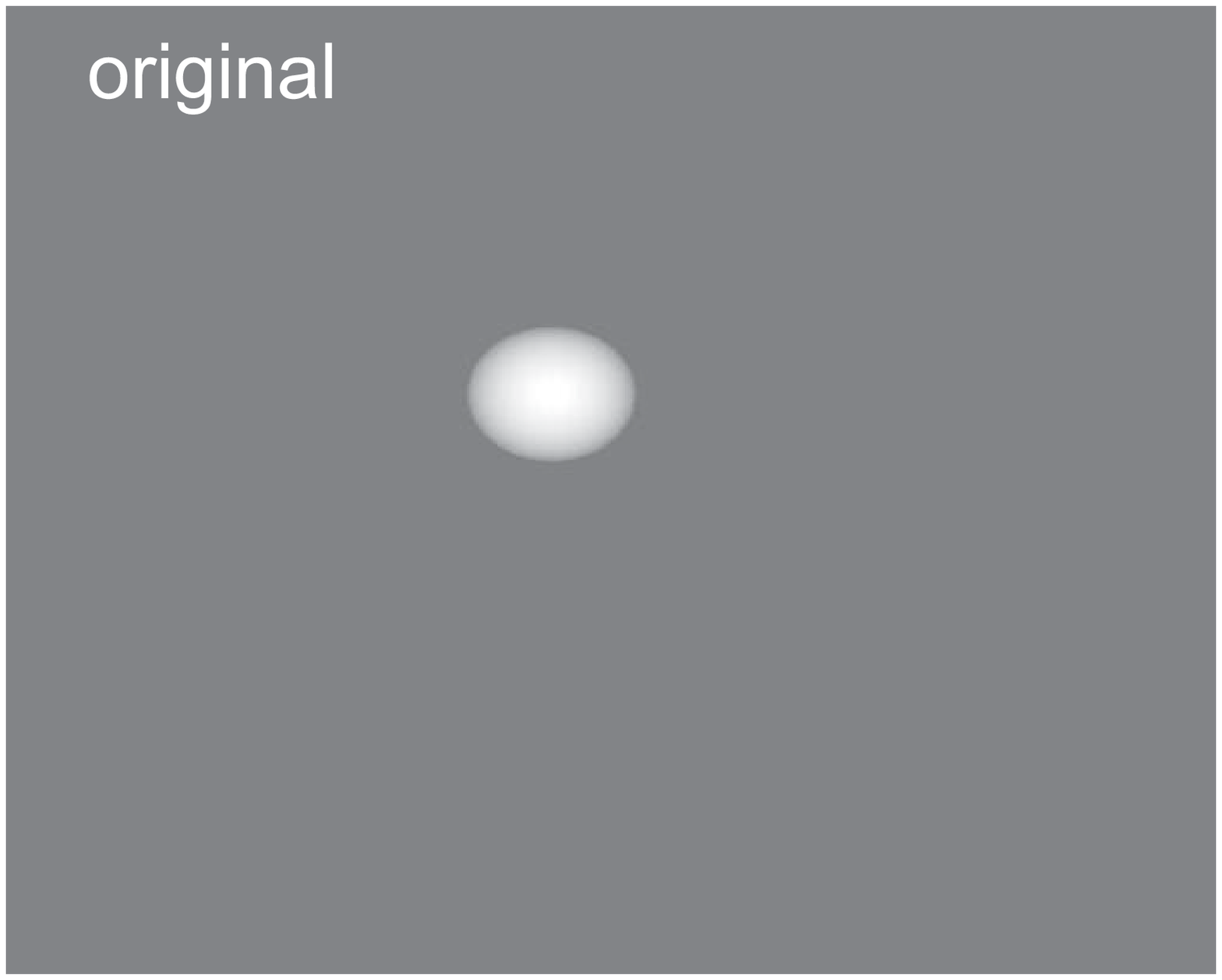}\quad
\includegraphics[width=0.325\textwidth,height=0.3\textwidth]{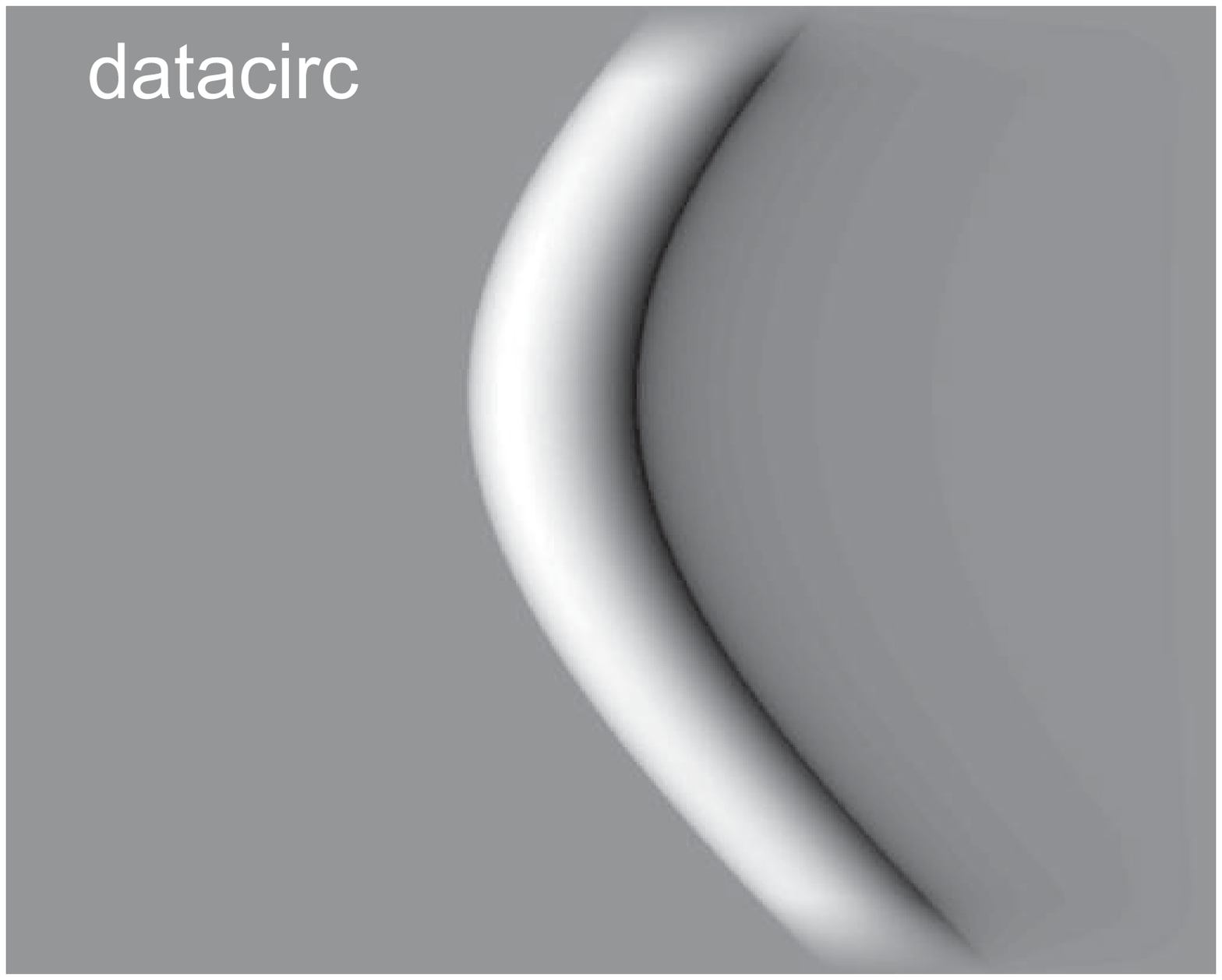}\\[1em]
\includegraphics[width=0.325\textwidth,height=0.3\textwidth]{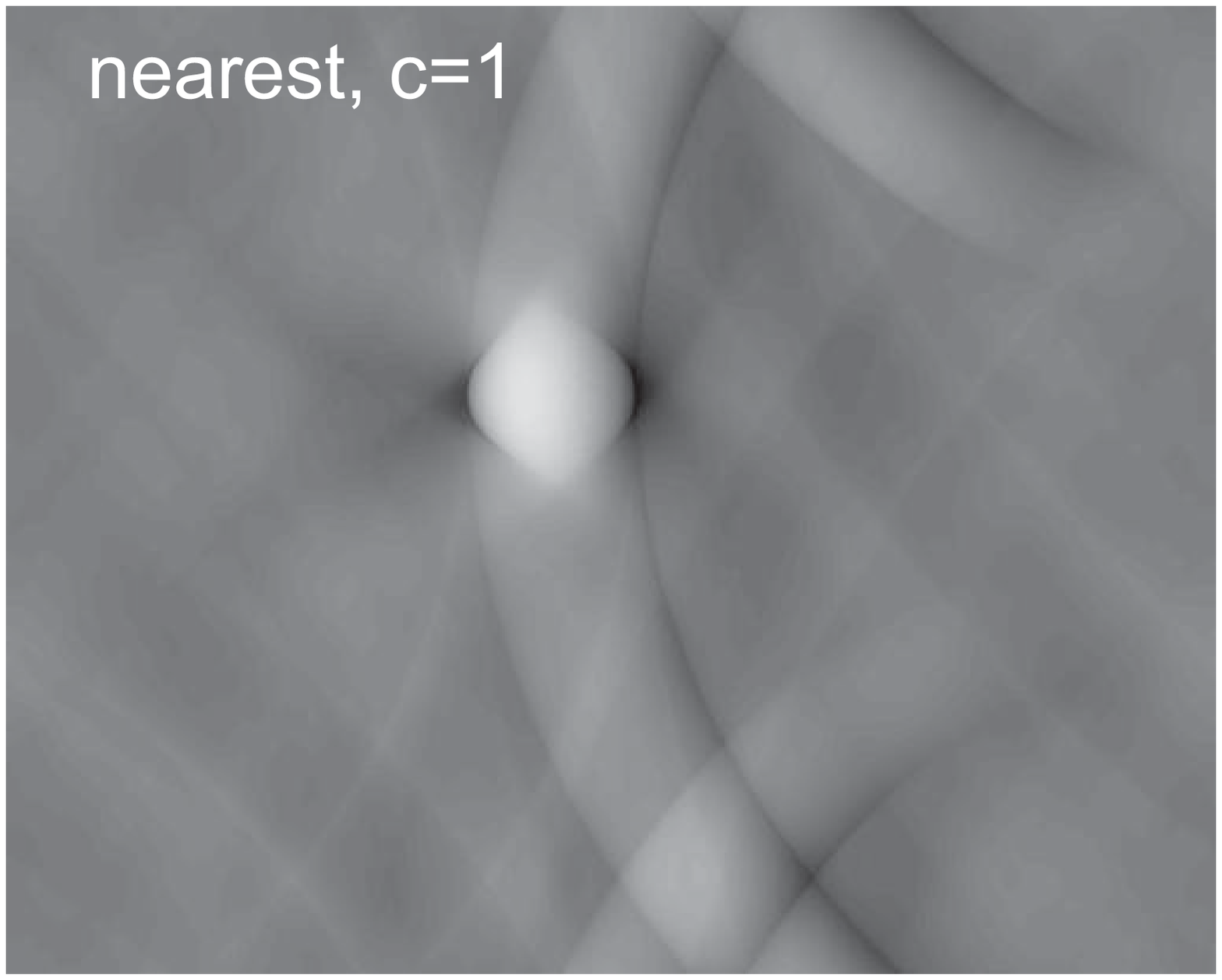}\quad
\includegraphics[width=0.325\textwidth,height=0.3\textwidth]{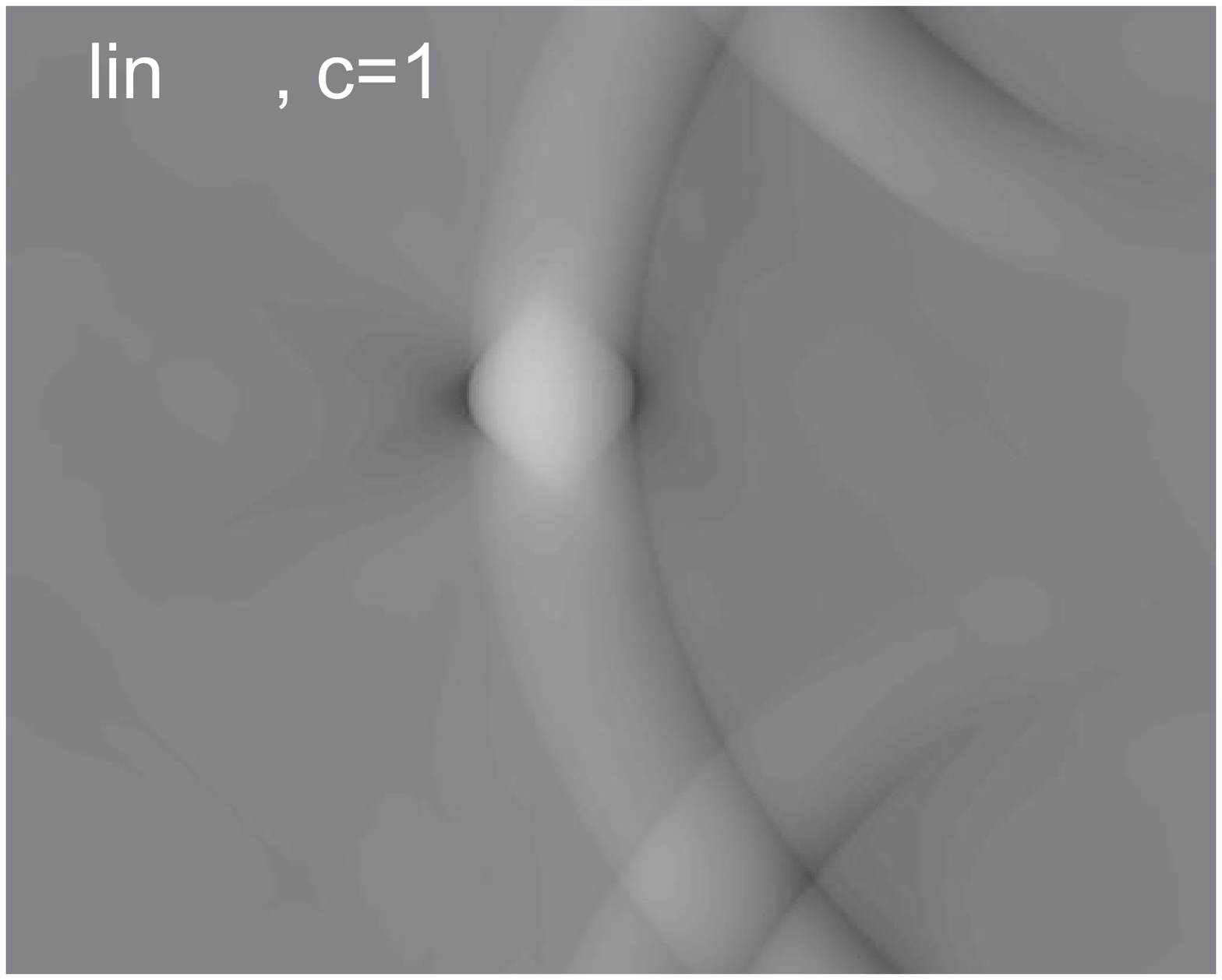} \\[1em]
\includegraphics[width=0.325\textwidth,height=0.3\textwidth]{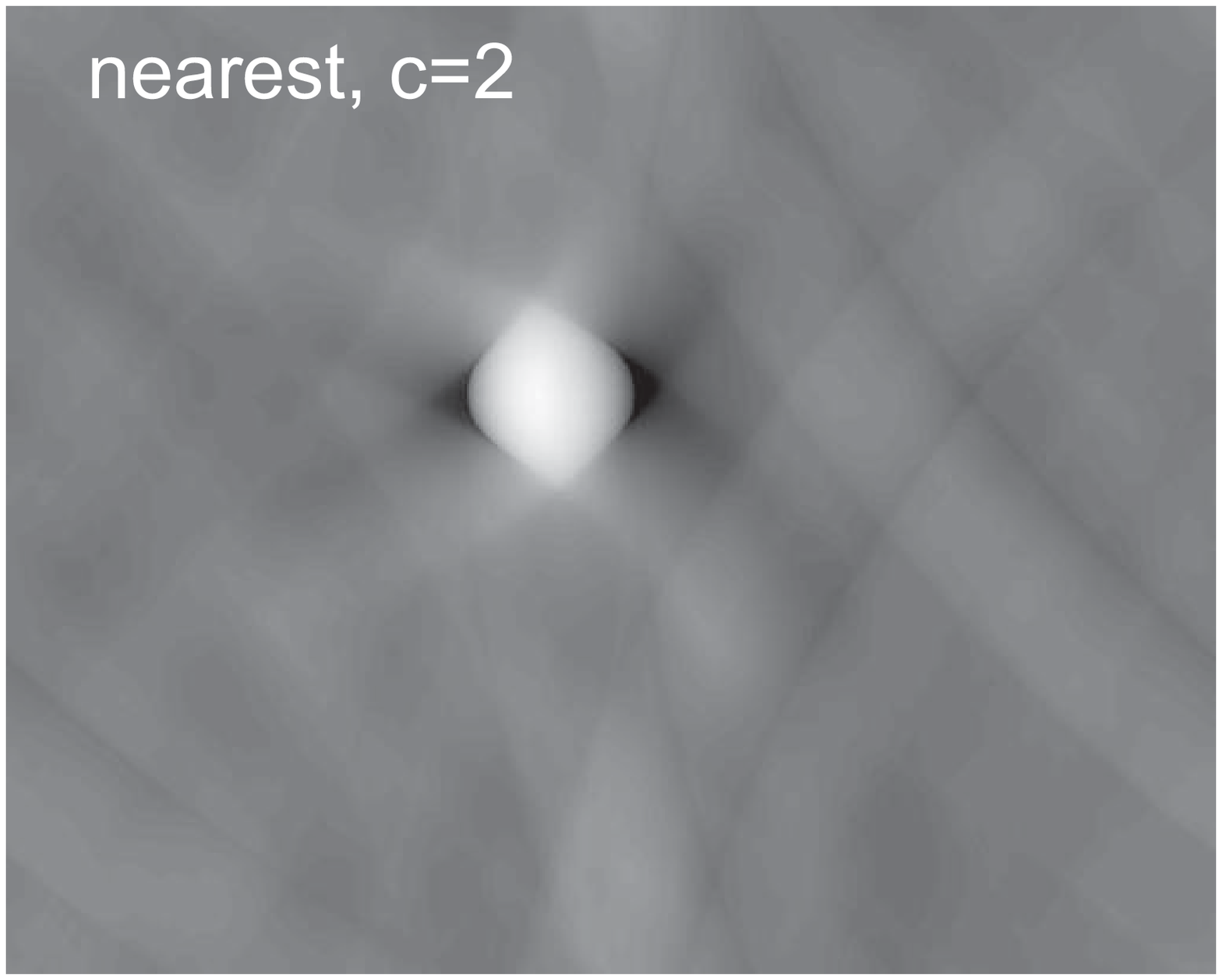}\quad
\includegraphics[width=0.325\textwidth,height=0.3\textwidth]{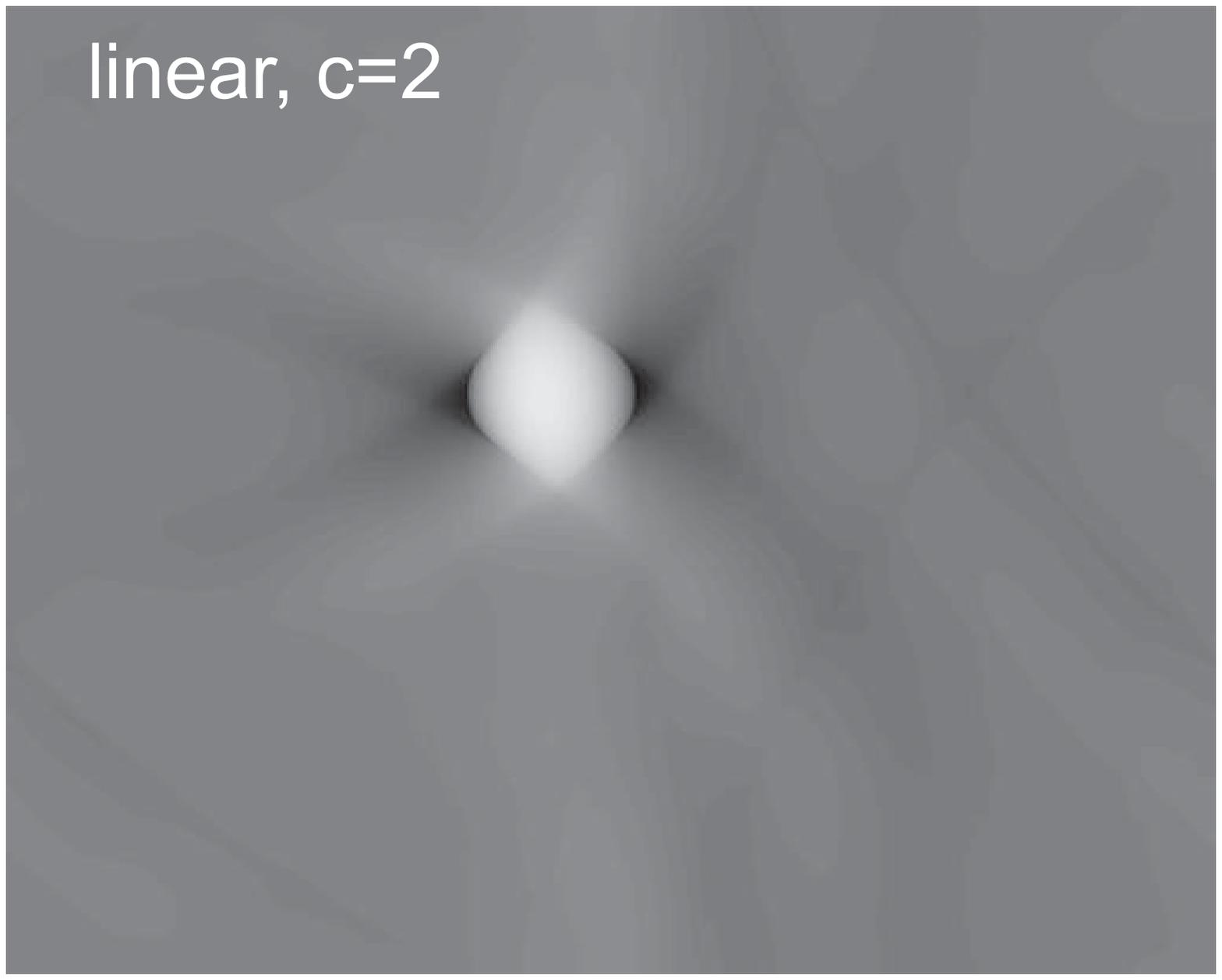}
\caption{\textbf{Reconstruction  with interpolation based Fourier algorithms.}
White corresponds to function value 1, black to  function value -0.4.
\emph{Top Line:} Phantom and analytic data.
\emph{Middle line:} Reconstruction without oversampling ($c=1$).
\emph{Bottom line:} Reconstruction with  oversampling  ($c=2$).}
\label{fg:fcirc-1}
\end{figure}

\begin{figure}[t!]
\centering
\includegraphics[width=0.325\textwidth,height=0.3\textwidth]{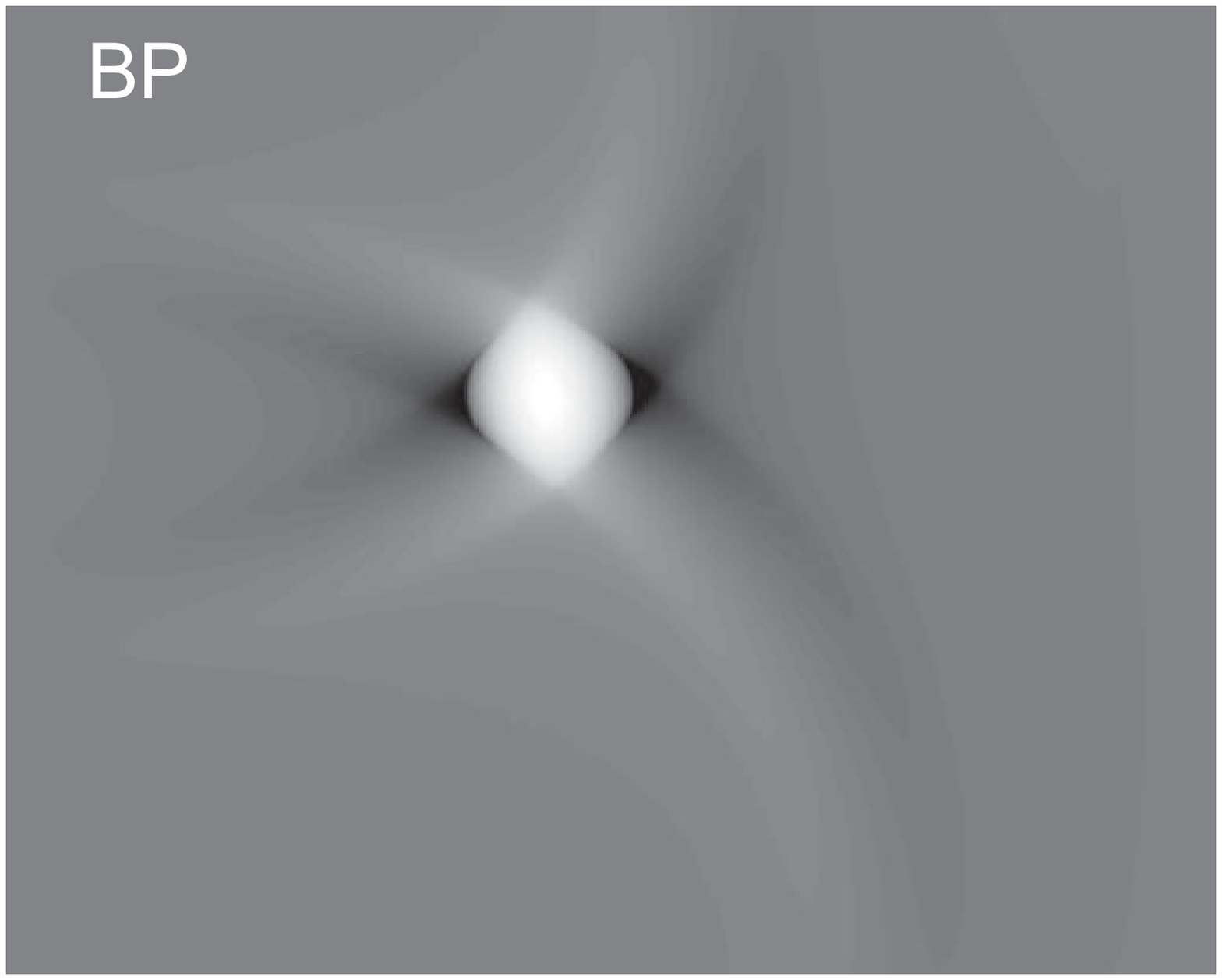}\quad
\includegraphics[width=0.325\textwidth,height=0.3\textwidth]{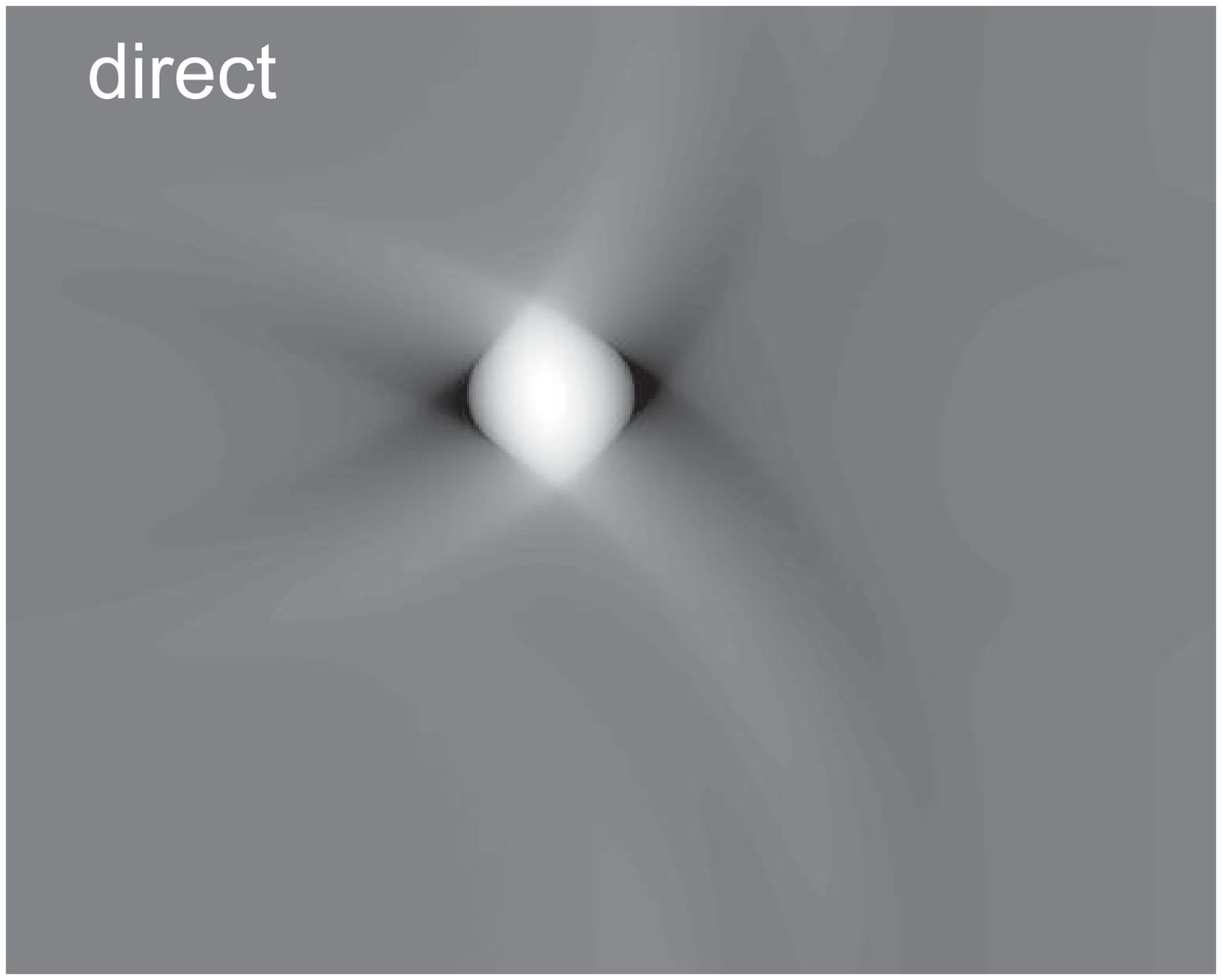} \\[1em]
\includegraphics[width=0.325\textwidth,height=0.3\textwidth]{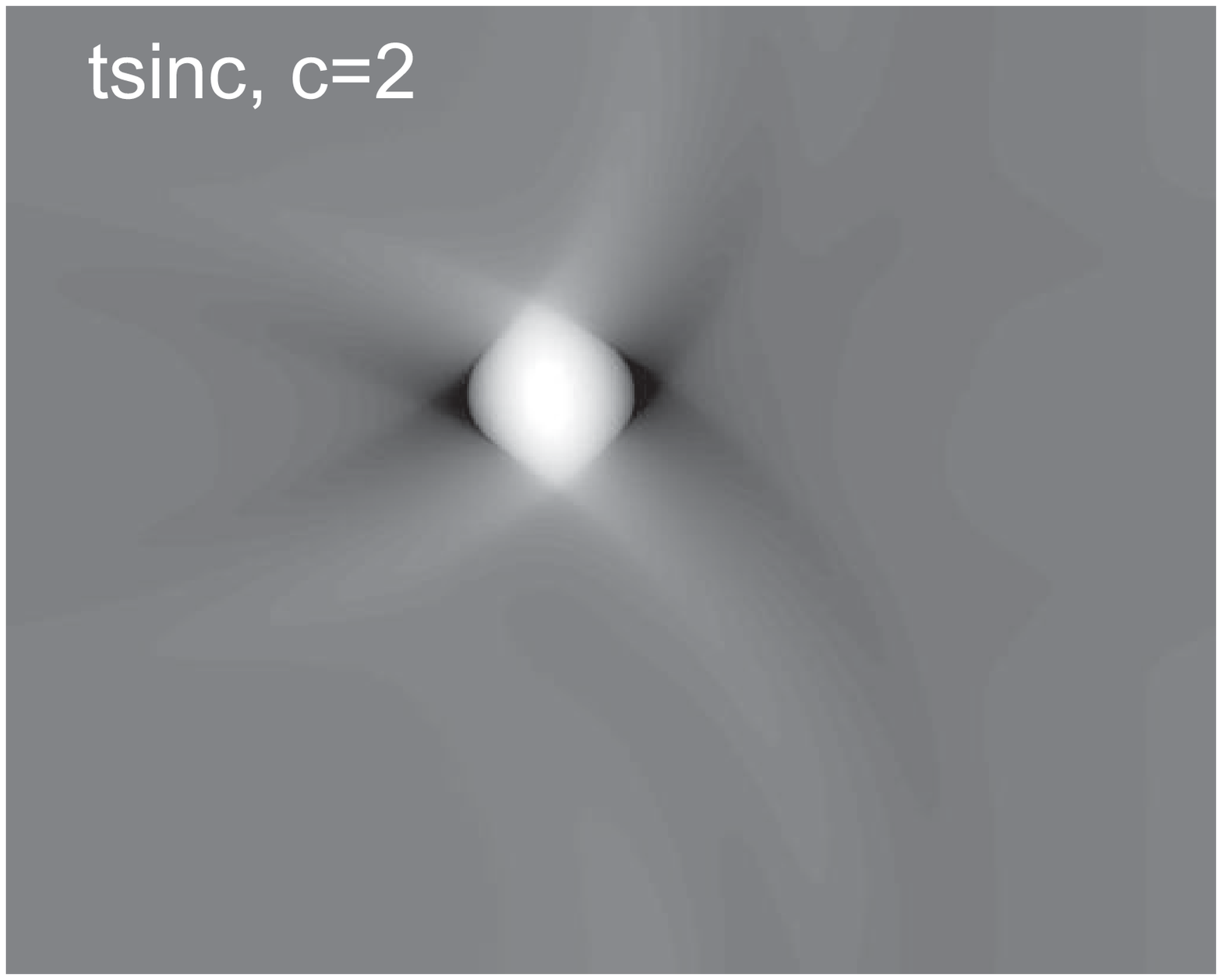}\quad
\includegraphics[width=0.325\textwidth,height=0.3\textwidth]{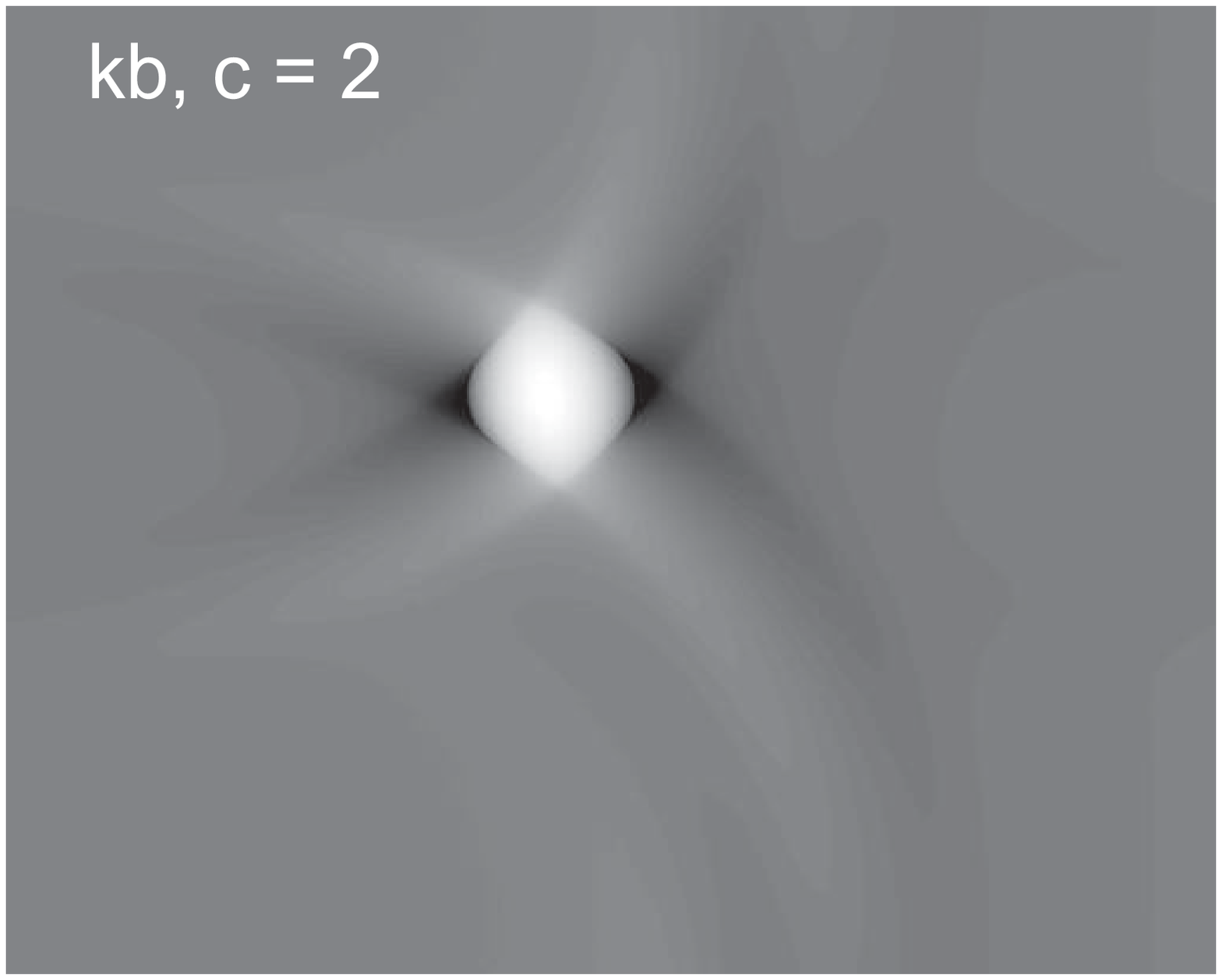} \\[1em]
\includegraphics[width=0.325\textwidth,height=0.3\textwidth]{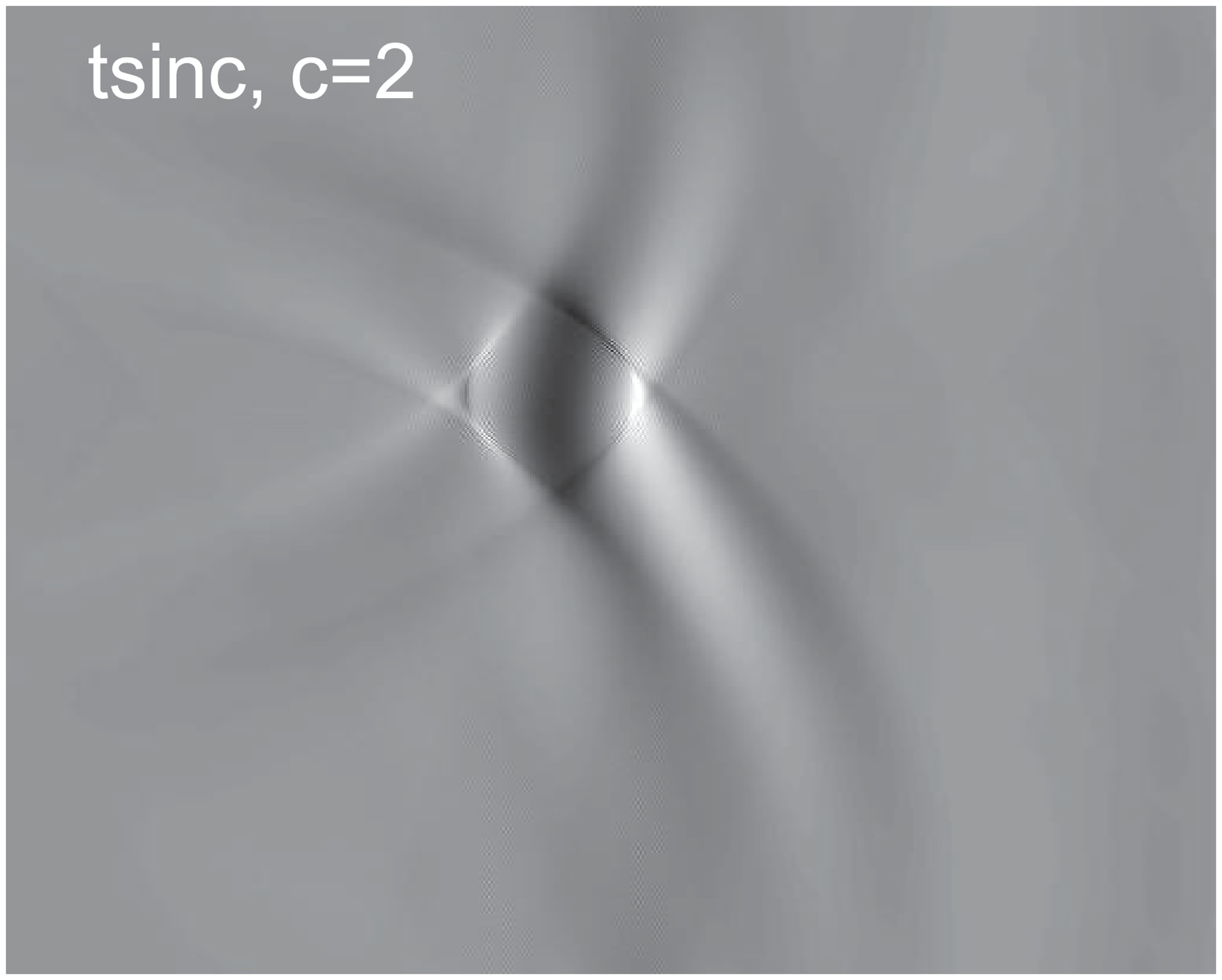} \quad
\includegraphics[width=0.325\textwidth,height=0.3\textwidth]{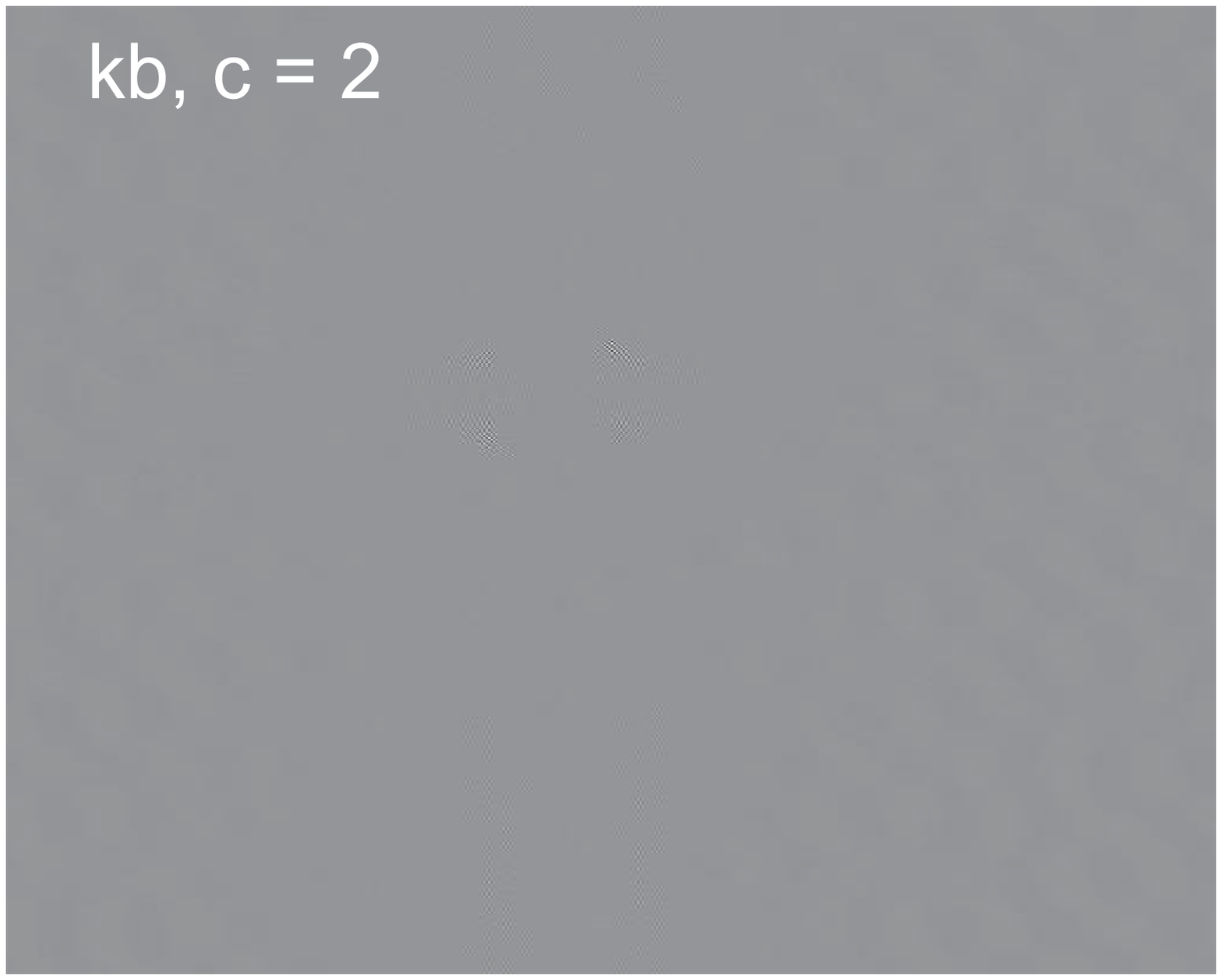}
\caption{\textbf{Improved reconstructions.}
\emph{Top Line:} Back projection (left) and direct reconstruction (right).
\emph{Middle Line:} Truncated $\sinc$  (left) and  nonuniform FFT based
Fourier algorithm (right). Here white corresponds to function value 1, black to  function value -0.4.
\emph{Bottom Line:}
Difference images between direct and truncated $\sinc$  reconstruction (left),
and direct and nonuniform FFT based reconstruction (right).
Here white (resp. black) corresponds to function value $0.04$ (resp. $-0.04$).
} \label{fg:fcirc-2}
\end{figure}

\subsection{Circular shaped object}

As first case example we use a circular shaped object
\[
    f_{\rm circ}(\f x) = \frac{2}{a}  \left\{
                  \begin{array}{ll}
                    \bigl(a^2 - \abs{\f x-\f x_0}^2\bigr)^{1/2},
                    & \text{ if } \abs{\f x-\f x_0} < a\,, \\
                    0, & \text{ otherwise}\,,
                  \end{array}
                \right.
\]
centered at $\f x_0 := (x_0, y_0)$ (see top left image in Figure \ref{fg:fcirc-1}). For such a simple object reconstruction
artifacts can be identified very clearly. Moreover, the data $\Qo f_{\rm circ}$ can be
evaluated analytically (see \cite[Equation (B.1)]{BurBauGruHalPal07}) as
\begin{equation*}
    ( \Qo f_{\rm circ}) ( x, t)
    =
    \frac{1}{a} \ \mathrm{Re}  \left[ (s_+ - s_-) -  t \log \left( \frac{s_+ + ( t + a_i)}{s_- + (t - a_i) } \right) \right] .
\end{equation*}
Here $s_\pm := ( (t \pm a)^2 + \abs{(x, 0)-\f x_0}^2  )^{1/2}$,
$\log(\cdot)$ is the principal branch  of the complex logarithm, and
$\mathrm{Re} [z]$ denotes the real part of complex number $z$. The
reconstruction results are depicted in Figures \ref{fg:fcirc-1} and
\ref{fg:fcirc-2}. Table \ref{tb:performance} and Figure \ref{fg:performance}
compare run times with the relative $\ell^2$-error
 \[
 \frac{\norm{\f f - \f f^\dag}_{\ell^2} }{ \norm{\f f^\dag}_{\ell^2} }
 =
\frac{\left( \sum_{m,n}(f_{m,n} - f^\dag_{m,n})^2 \right)^{1/2} }
{\left( \sum_{m,n}(f^\dag_{m,n})^2 \right)^{1/2}} \,,
 \]
were $ \f f^\dag =(f^\dag_{m,n})$ denotes the discrete image  obtained by
direct Fourier reconstruction. Run times were measured for Matlab
implementations on a personal computer with 2.4 GHz Athlon processor.

In order to demonstrate the stability of the Fourier algorithms, we also performed
reconstructions from noisy data, where Gaussian noise was added with a variance equal to $20\%$ of
the maximal data value. The reconstruction results are depicted in  Figure
\ref{fg:fcirc-3}.

\begin{table}[b!]
\centering
\begin{tabular}{l | @{\hspace{1cm}} c  @{\hspace{1cm}} c  @{\hspace{0.5cm}} c  }
\toprule
&    $c$  & $\ell^2$-error & runtime (sec)  \\
\midrule
back projection                       & -                  & -   & 88.9         \\
direct reconstruction       & -                  & -  & 54.1          \\
nearest neighbor            & 1                       & 0.75\phantom{6}         & 0.65       \\
nearest neighbor            & 2                       & 0.40\phantom{6}         & 0.85       \\
linear                      & 1                     & 0.65\phantom{6}         & 0.8        \\
linear                      & 2                     & 0.21\phantom{6}         & 0.95        \\
Truncates $\sinc$           & 2                        & 0.04\phantom{6}       & 1.6        \\
Kaiser Bessel               & 2                      & 0.006                   & 1.6        \\
\bottomrule
\end{tabular}
\caption{Run times and error  of different reconstruction methods.} \label{tb:performance}
\end{table}

\begin{psfrags}
\psfrag{linear}{linear}
\psfrag{Kaiser Bessel}{Kaiser Bessel}
\psfrag{nearest}{nearest}
\psfrag{runtime}{\hspace{-0.03\textwidth}reconstruction time}
\psfrag{log}{\hspace{-0.03\textwidth}$\log_{10}\bigl( \norm{\f f - \f f^\dag}_{\ell^2} \bigl/ \norm{\f f^\dag}_{\ell^2} \bigr)$}
\psfrag{10}{}
\begin{figure}[t!]
\centering
\includegraphics[width=0.85\textwidth]{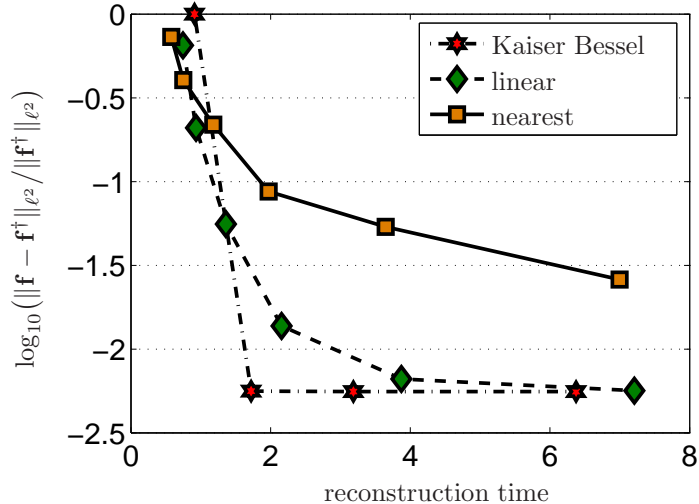}
\caption{\textbf{Reconstruction time versus error.}
The points on the graphs  belong to runtimes and errors for reconstruction with
oversampling factors $c \in \set{ 1, 2, 4, 8, 16, 32}.$}
\label{fg:performance}
\end{figure}
\end{psfrags}

\begin{figure}[t!]
\centering
\includegraphics[width=0.325\textwidth,height=0.3\textwidth]{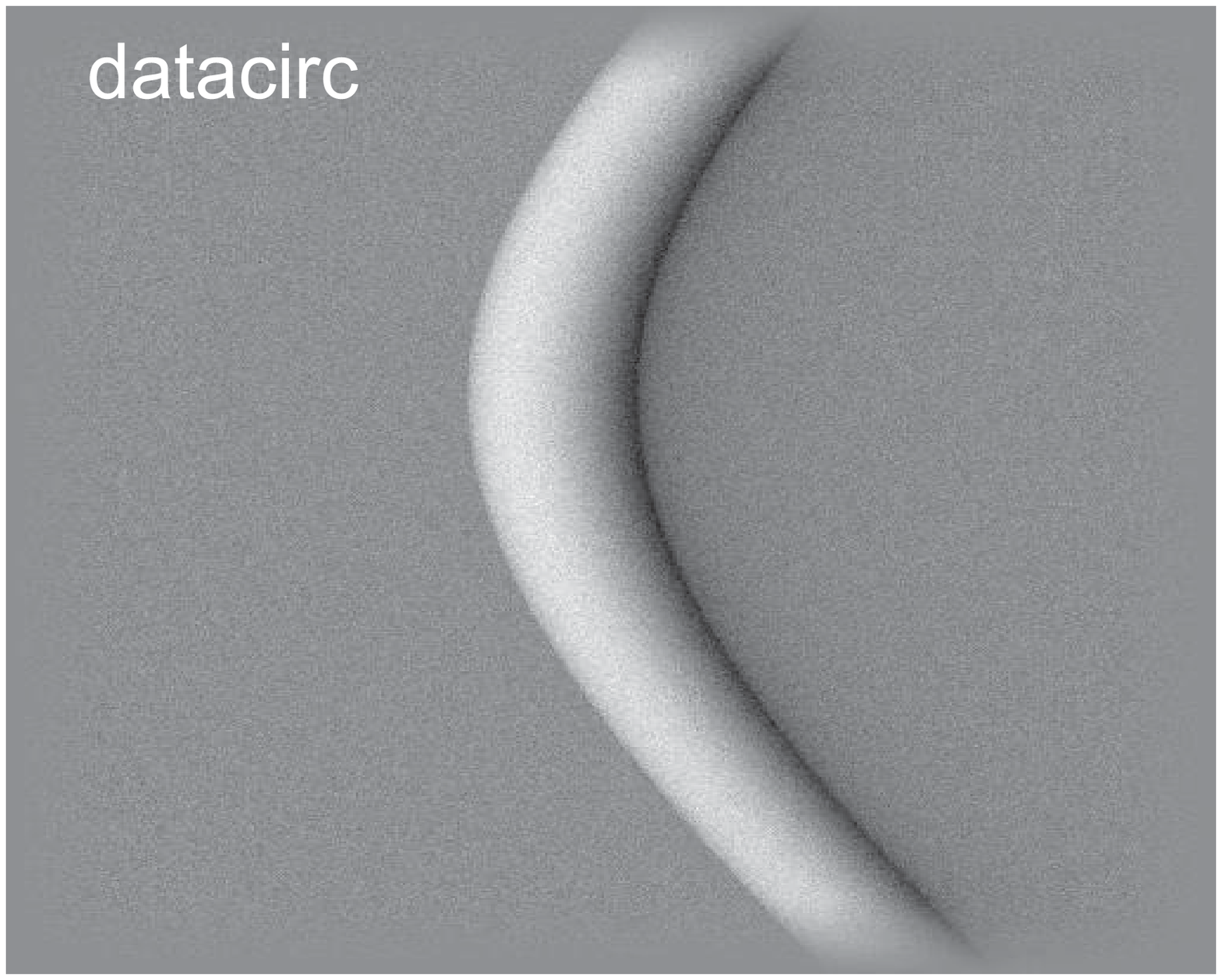} \quad\includegraphics[width=0.325\textwidth,height=0.3\textwidth]{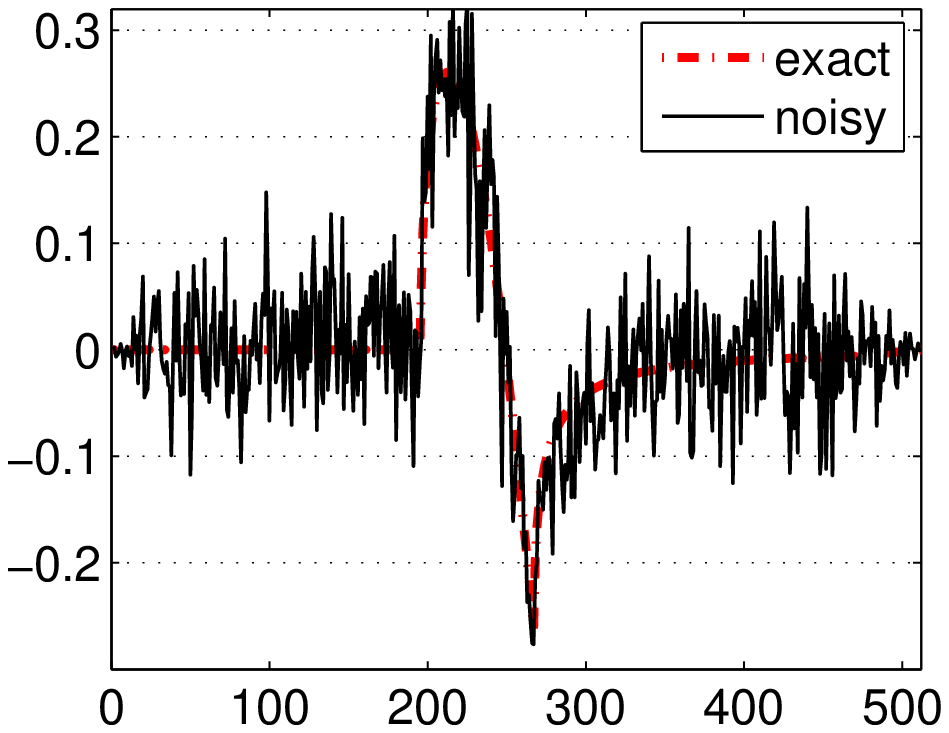}
\caption{ Noisy data used in Figure~\ref{fg:fcirc-3}}. 
\end{figure}

\begin{figure}[t!]
\centering
\includegraphics[width=0.325\textwidth,height=0.3\textwidth]{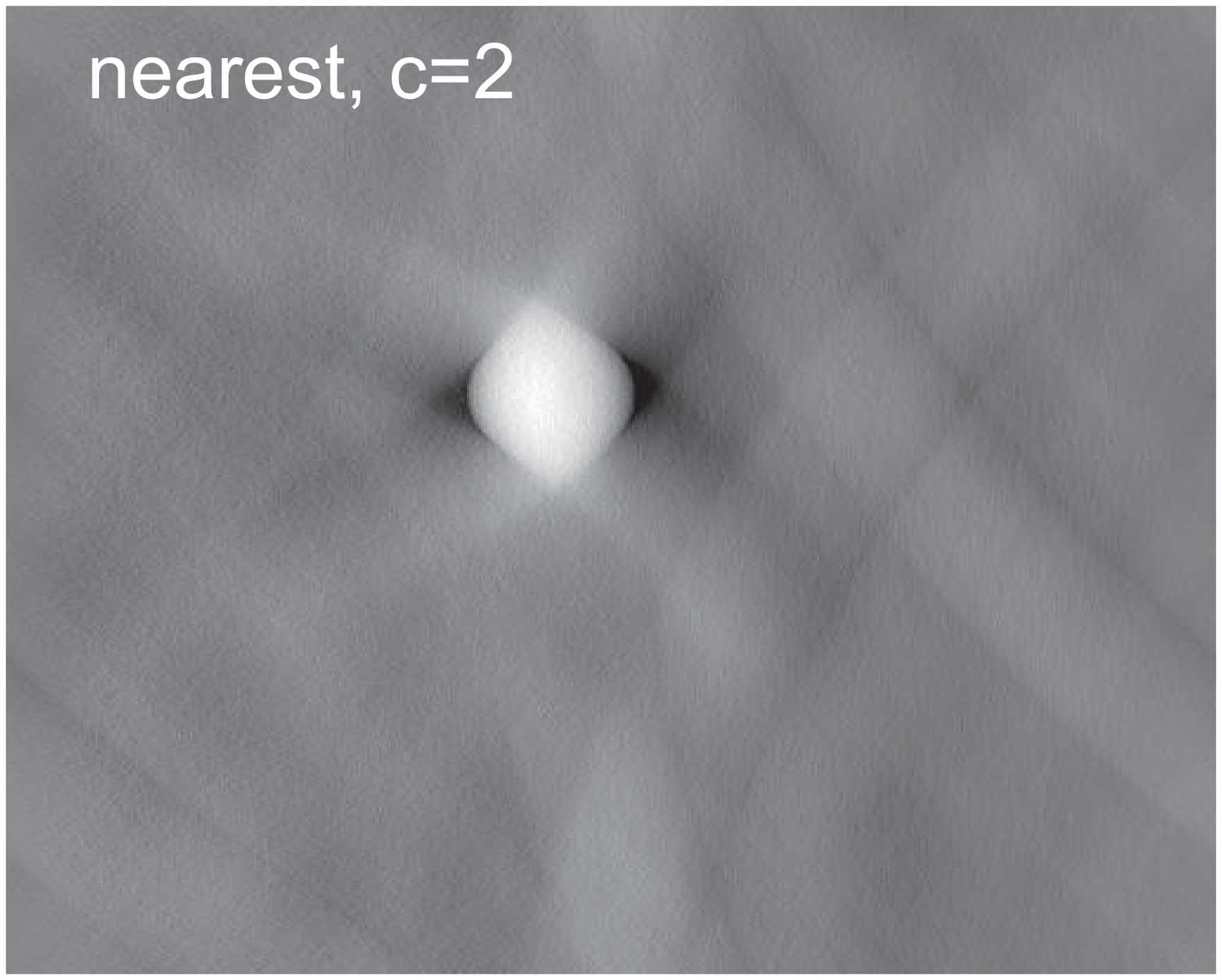}\quad
\includegraphics[width=0.325\textwidth,height=0.3\textwidth]{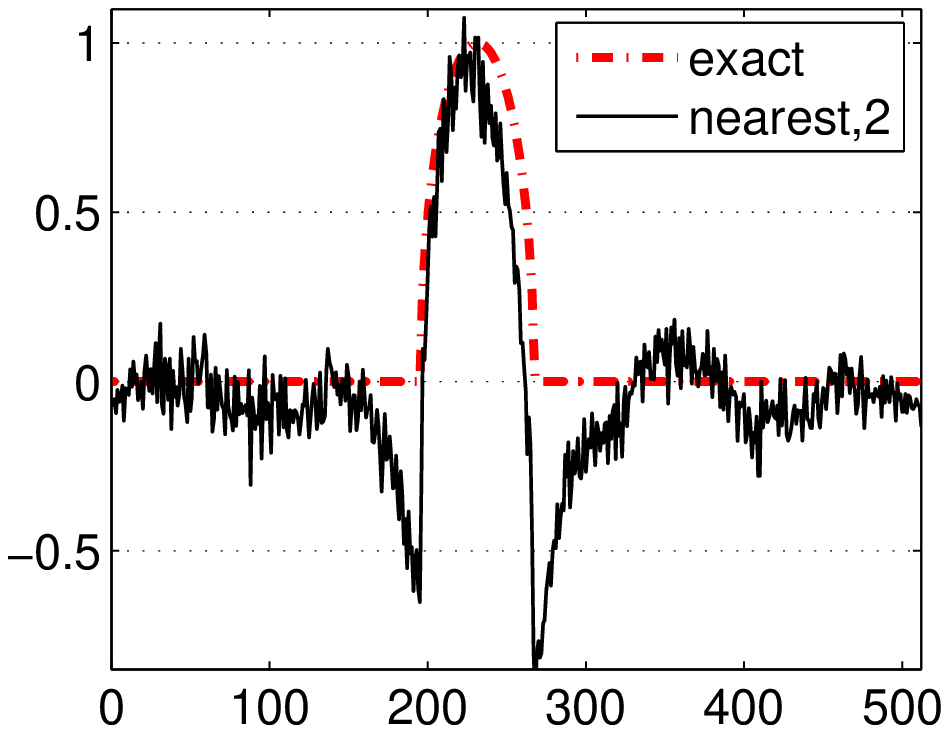}\\[1em]
\includegraphics[width=0.325\textwidth,height=0.3\textwidth]{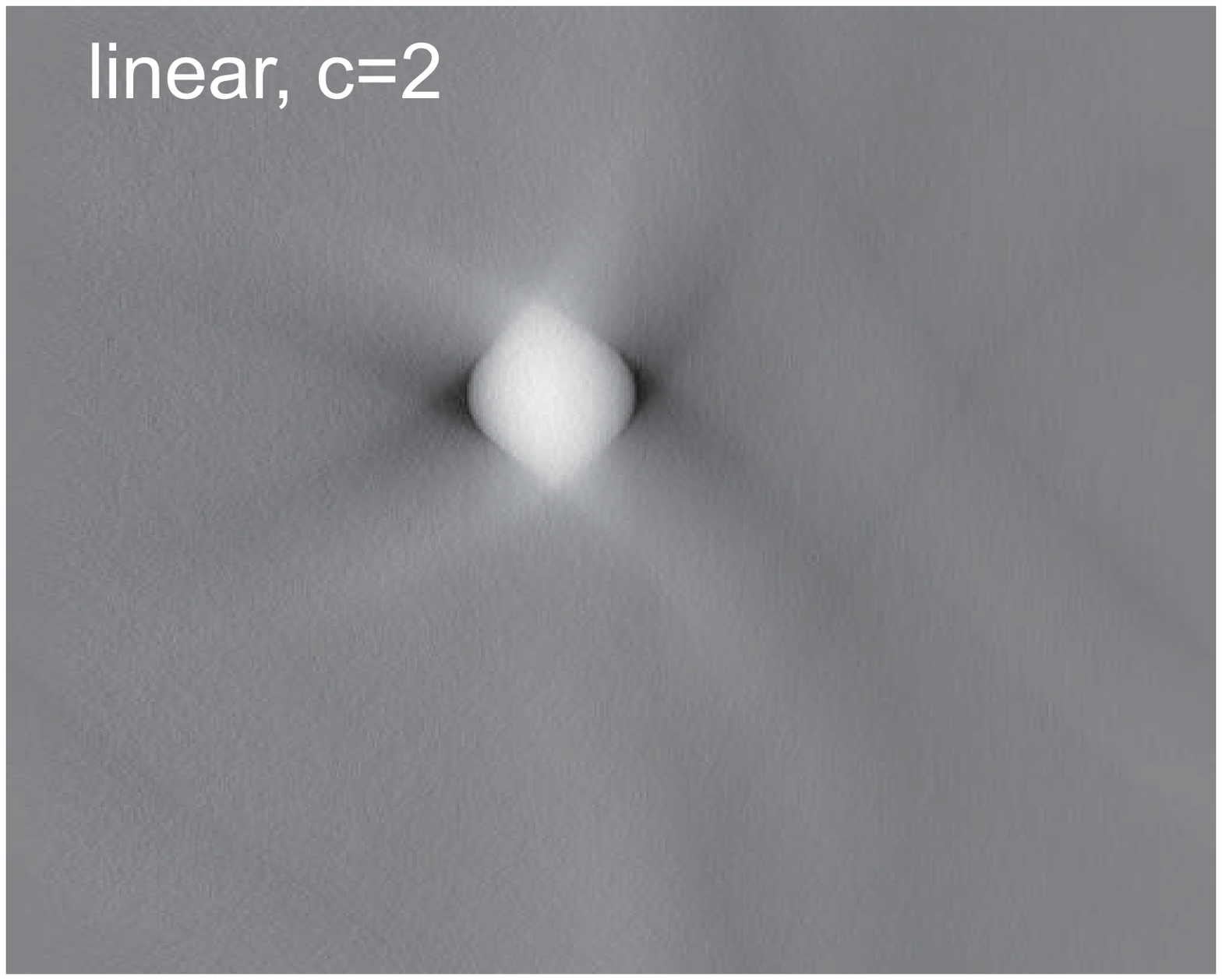}\quad
\includegraphics[width=0.325\textwidth,height=0.3\textwidth]{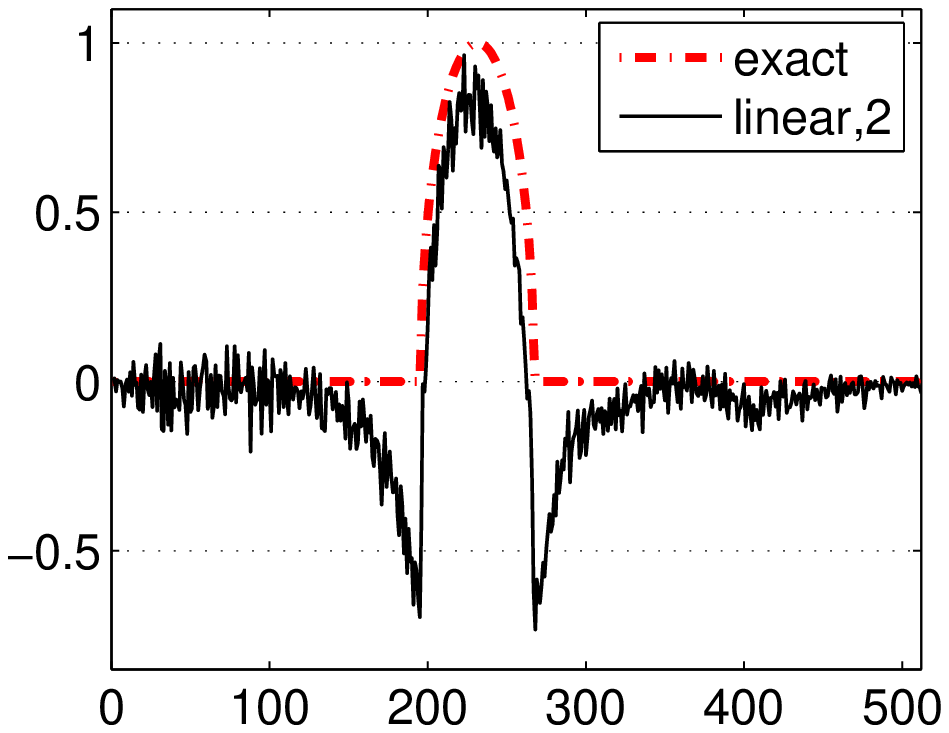}\\[1em]
\includegraphics[width=0.325\textwidth,height=0.3\textwidth]{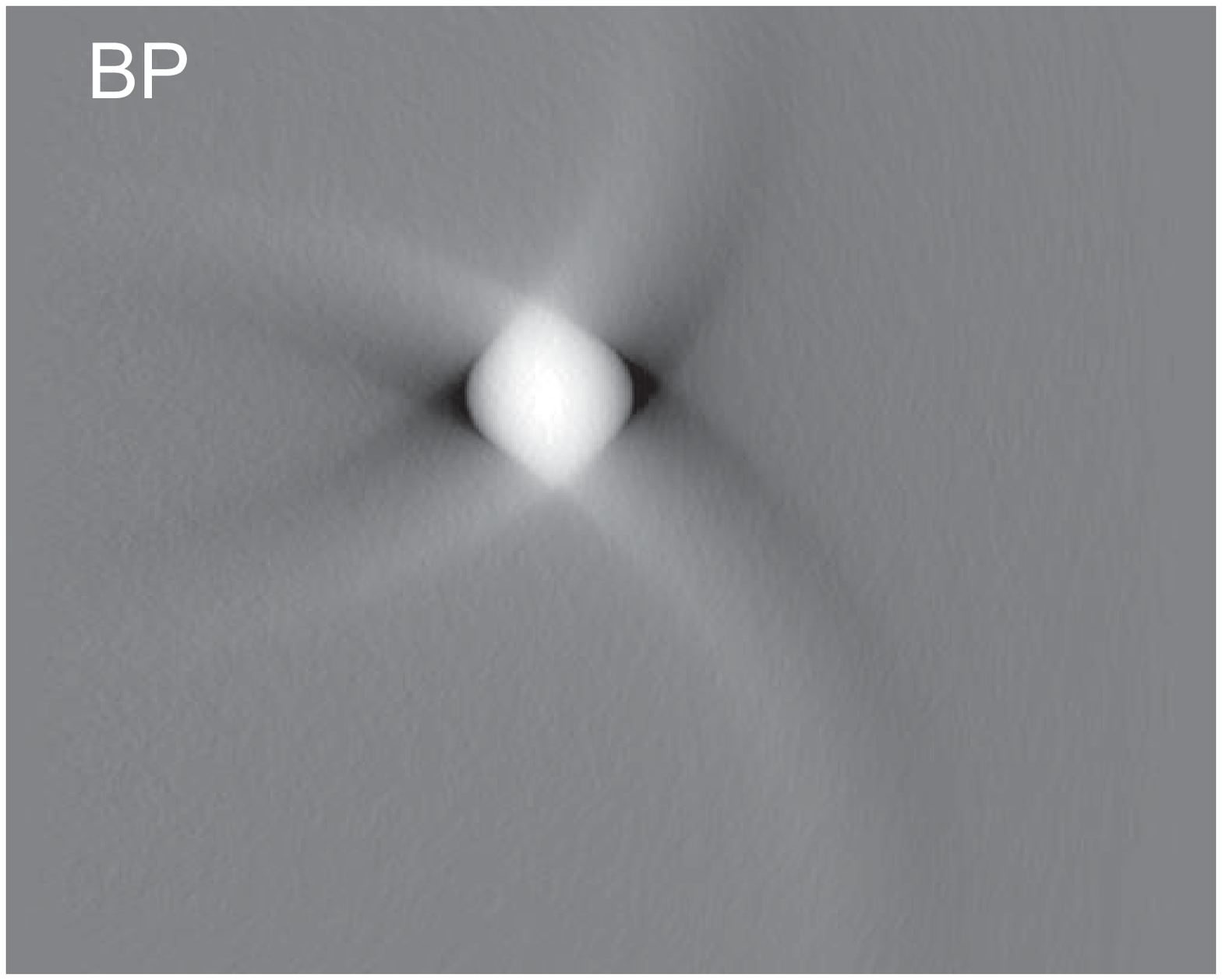}\quad
\includegraphics[width=0.325\textwidth,height=0.3\textwidth]{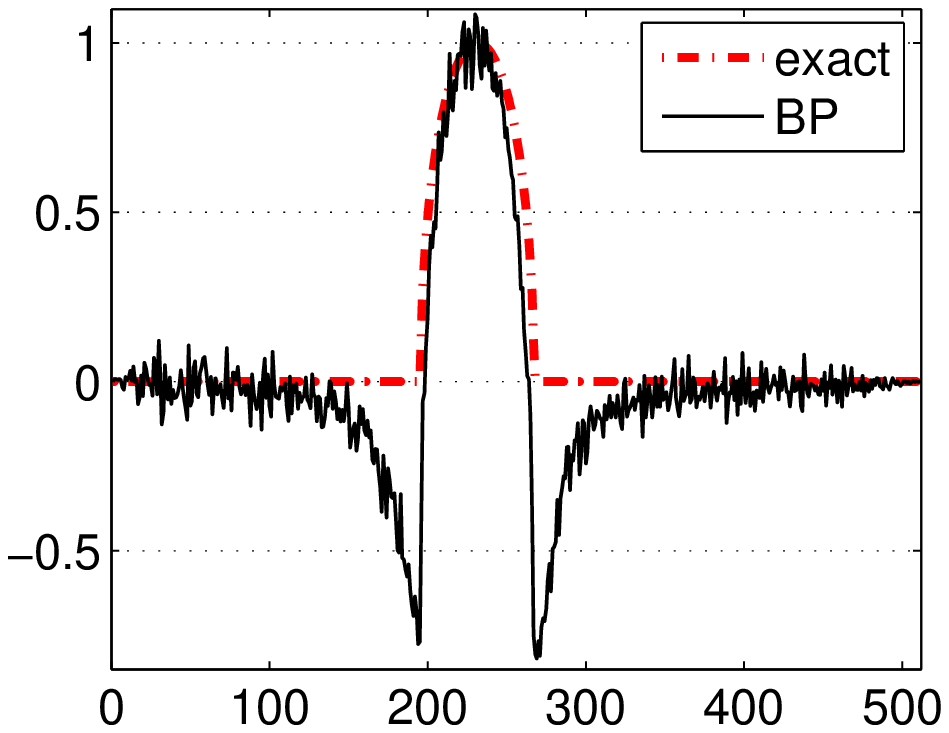}\\[1em]
\includegraphics[width=0.325\textwidth,height=0.3\textwidth]{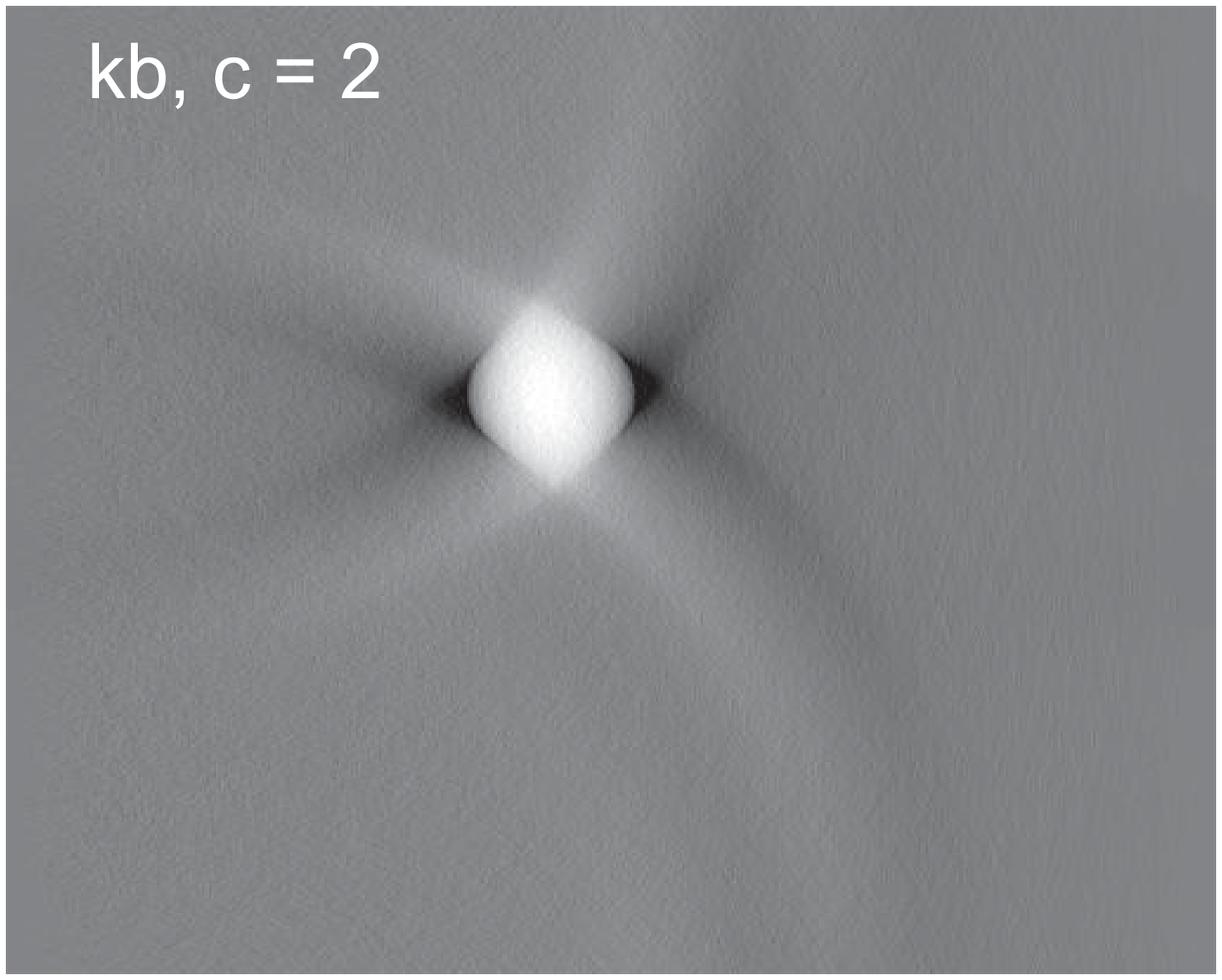}\quad
\includegraphics[width=0.325\textwidth,height=0.3\textwidth]{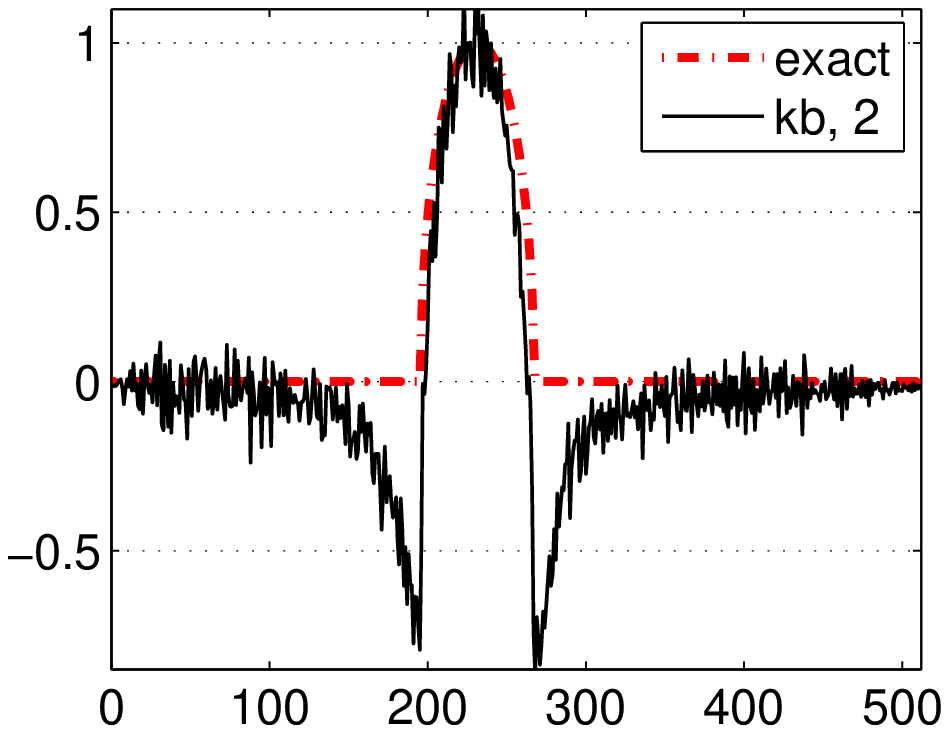}
\caption{ \textbf{Reconstruction from noisy data.}
\emph{Top  line:} Nearest neighbor interpolation based reconstruction ($c=2$).
\emph{Second line:} Linear interpolation based reconstruction ($c=2$).
\emph{Third line:} Reconstruction with back projection algorithm.
\emph{Bottom line:} Reconstruction with nonuniform FFT algorithm ($c=2$).
The horizontal profiles on the right are taken at $\set{x = x_0}$.}
\label{fg:fcirc-3}
\end{figure}

\begin{figure}[t!]
\centering
\includegraphics[width=0.325\textwidth,height=0.3\textwidth]{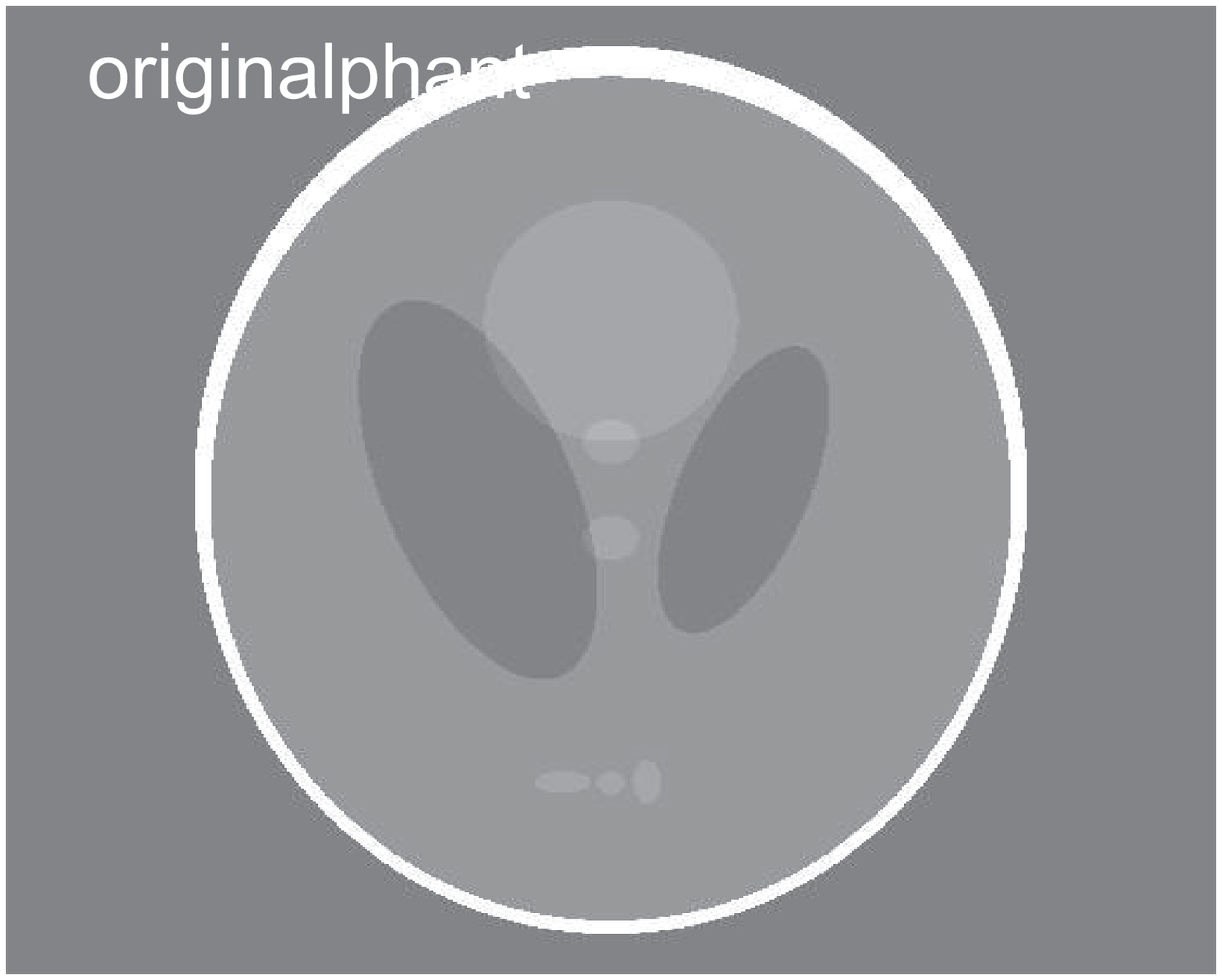}\quad
\includegraphics[width=0.325\textwidth,height=0.3\textwidth]{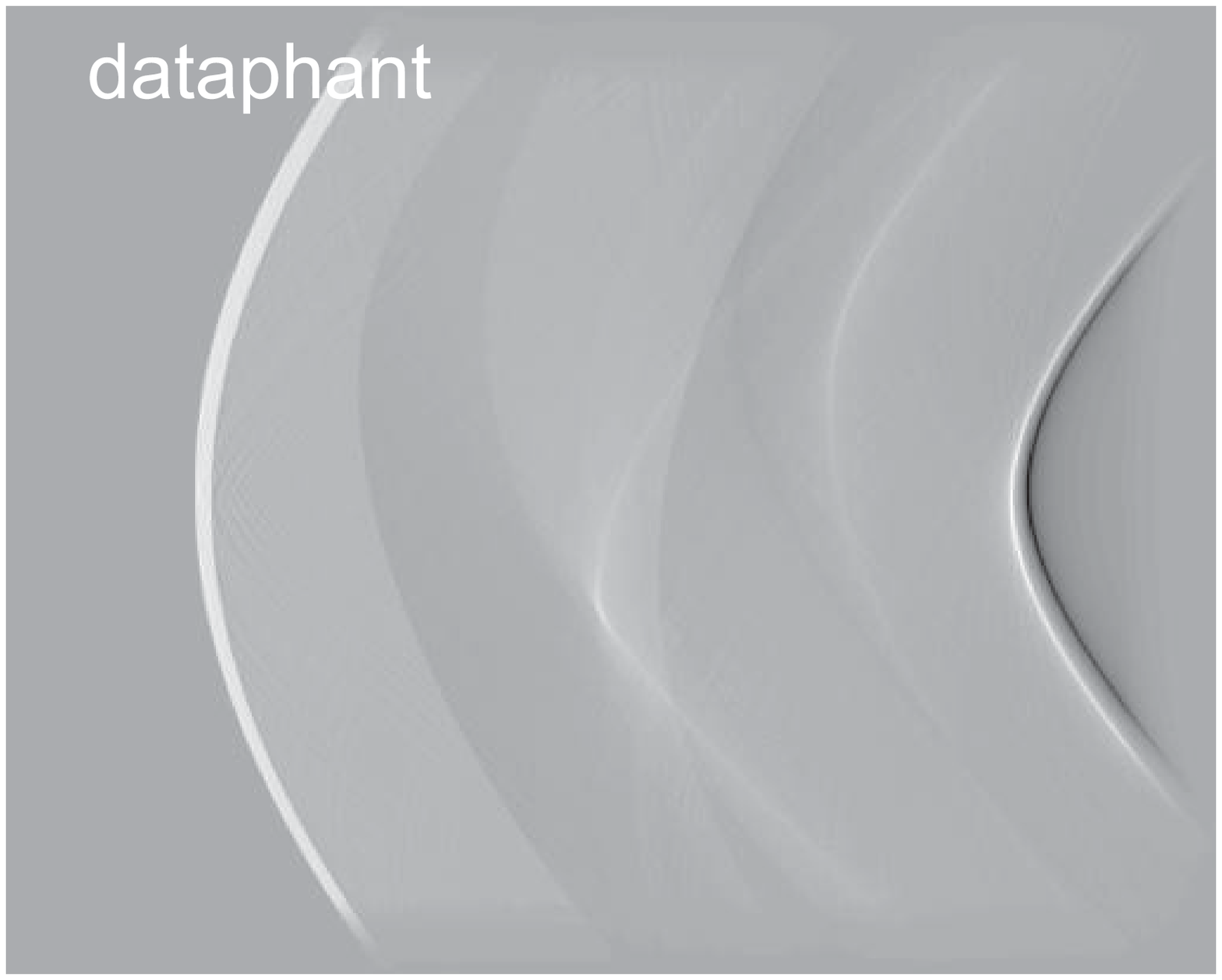}\\[1em]
\includegraphics[width=0.325\textwidth,height=0.3\textwidth]{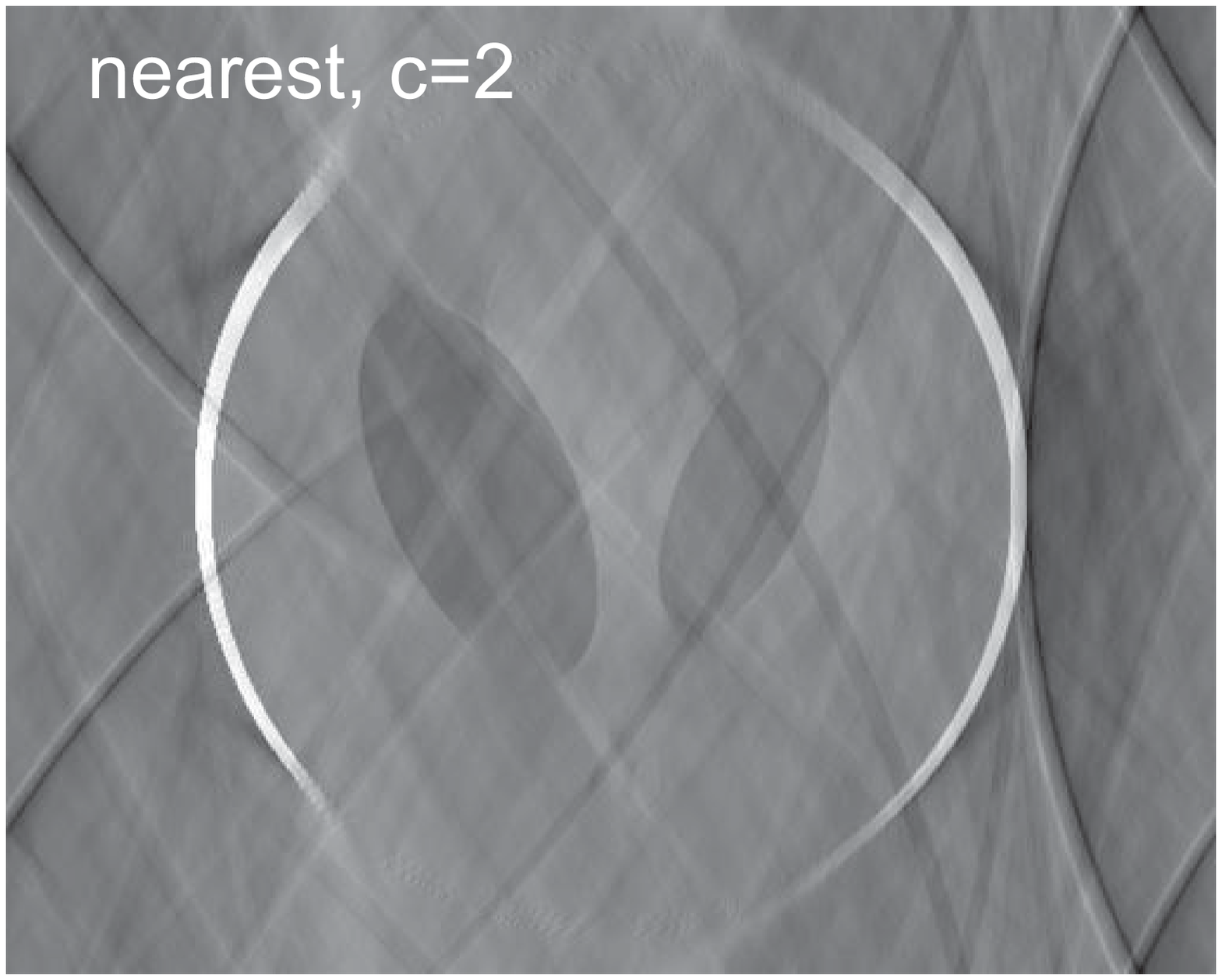}\quad
\includegraphics[width=0.325\textwidth,height=0.3\textwidth]{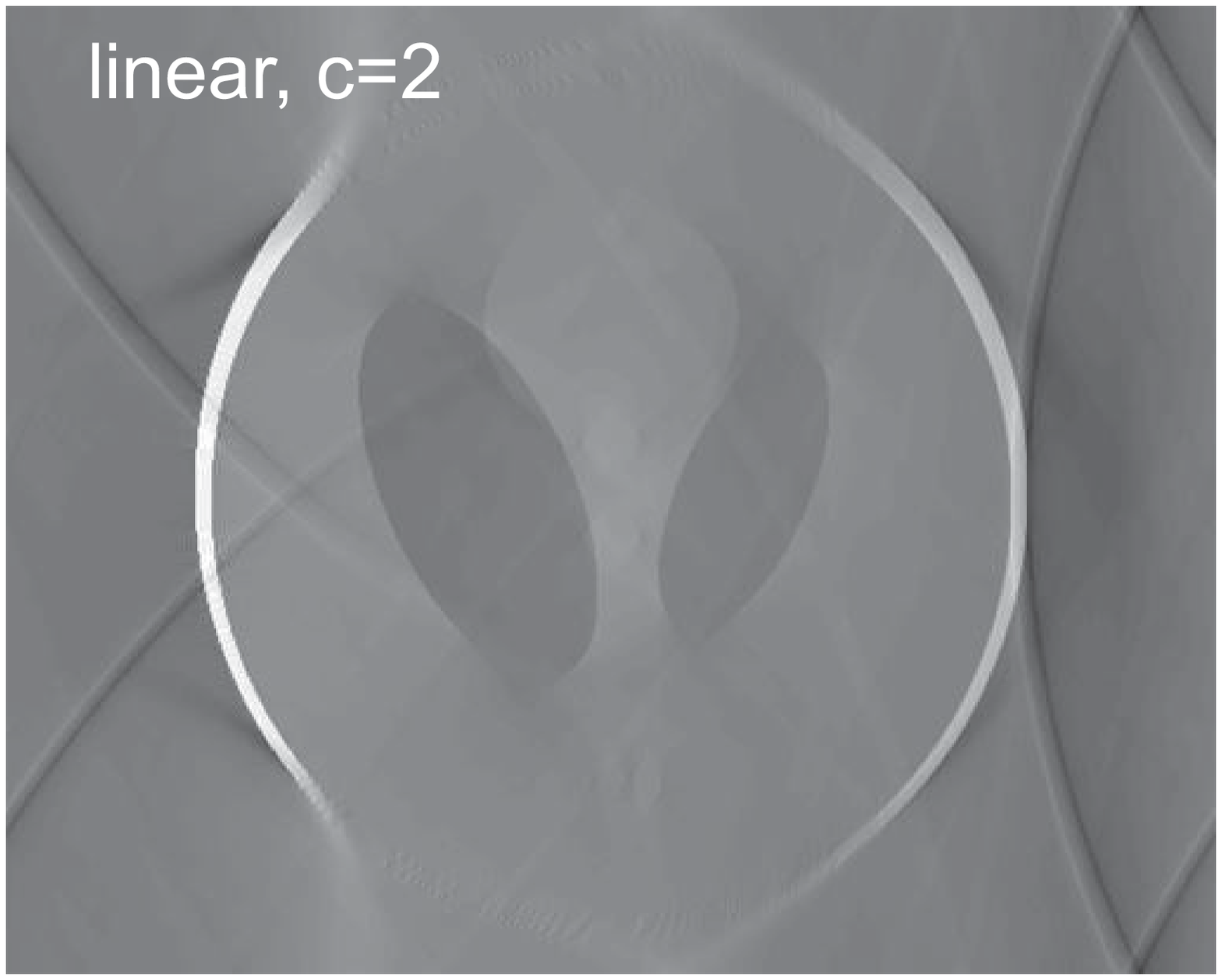}\\[1em]
\includegraphics[width=0.325\textwidth,height=0.3\textwidth]{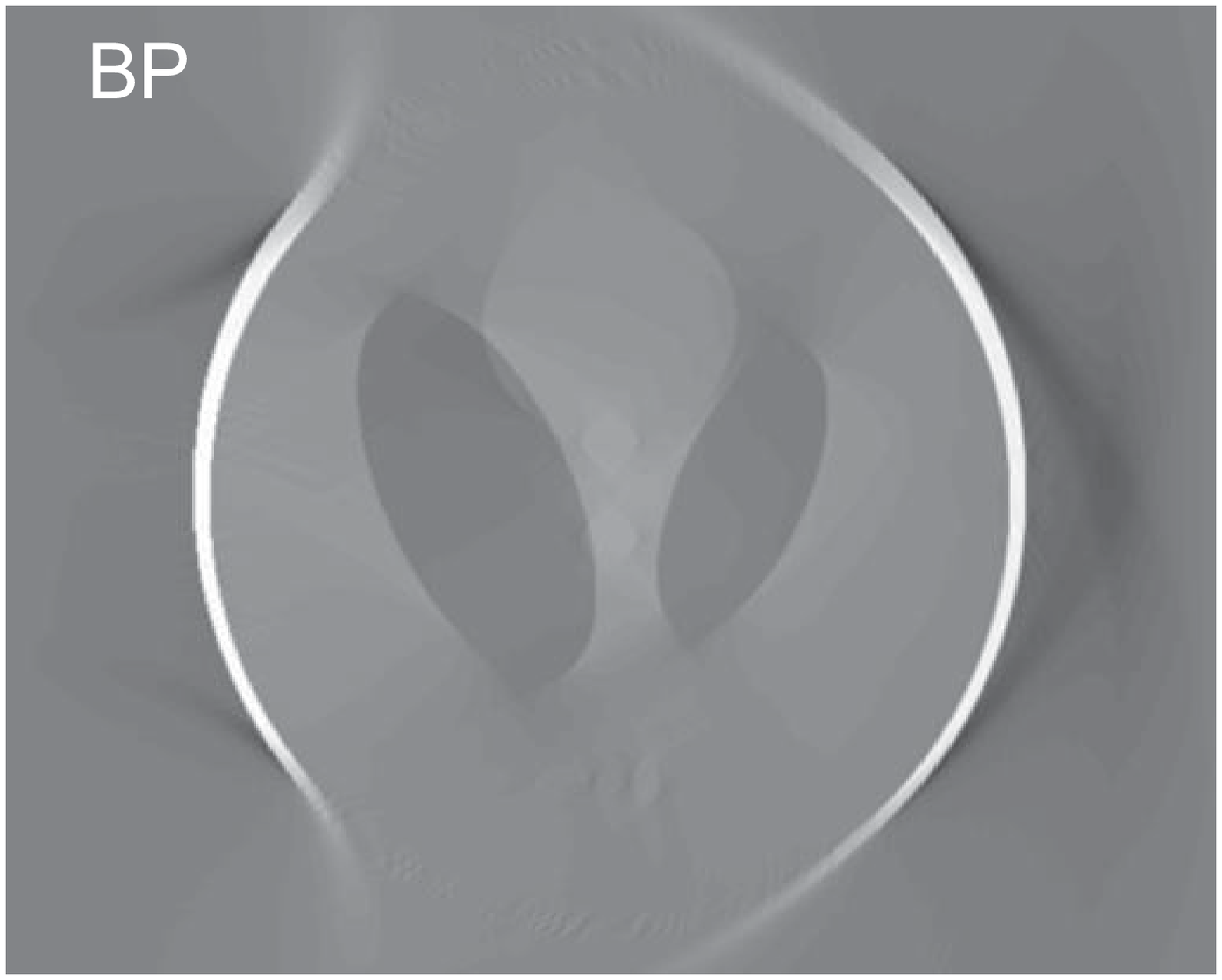}\quad
\includegraphics[width=0.325\textwidth,height=0.3\textwidth]{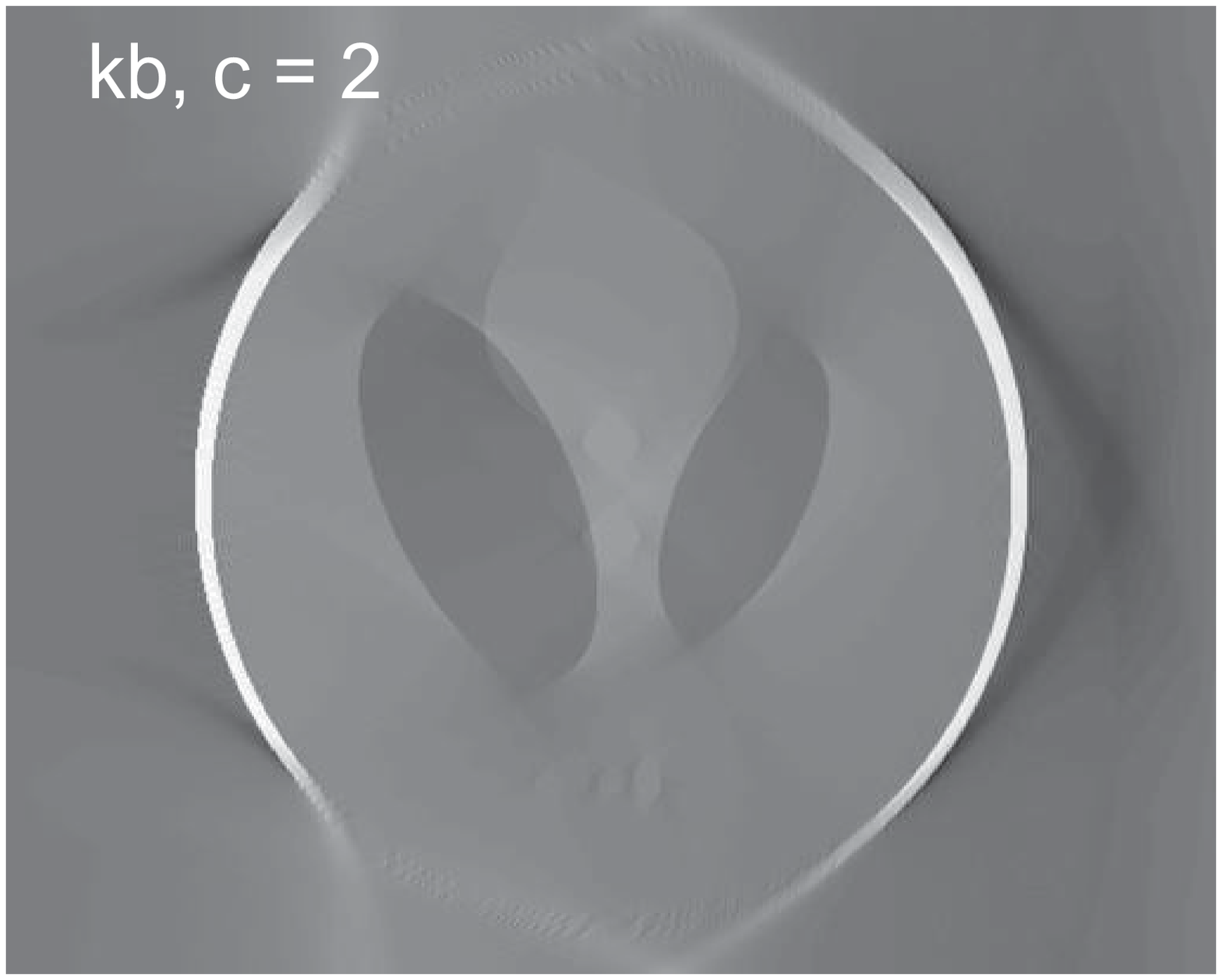}
\caption{\textbf{Reconstruction of Shepp Logan phantom.}
\emph{Top line:} Phantom and simulated data.
\emph{Second line :} Interpolation based reconstruction.
\emph{Bottom line:} Reconstruction with back projection algorithm (left) and
proposed nonuniform FFT algorithm (right).}
\label{fg:fphant}
\end{figure}

\subsection{Shepp--Logan  phantom}

In the next example we consider the Shepp--Logan  phantom
$f_{\rm phant}$, which is shown in top left image in Figure
\ref{fg:fphant}. The data were calculated numerically by
implementing d'Alemberts formula \cite{CouHil62},
\begin{equation*}
    (\Qo f_{\rm phant}) ( x, 0, t )
     =
    \partial_t \int_{0}^{t} \frac{ r (  \Mo f_{\rm phant})(x,0,r) }{\sqrt{t^{2}-r^{2}}} \ dr
\end{equation*}
with
\begin{equation*}
    (\Mo f_{\rm phant}) (x, 0, r) :=  \frac{1}{2\pi}
    \oint_{\abs{\sigma}  = 1} f_{\rm phant} \bigl( (x,0)+ r\sigma \bigr) d\sigma
\end{equation*}
denoting the spherical mean Radon transform of $f_{\rm circ}$.
The reconstruction results from simulated data are depicted in Figure
\ref{fg:fphant}.

\subsection{Discussion}

We emphasize that none of the above Fourier algorithms are designed to
calculate an approximation of $f$ but an approximation to the partial Fourier
reconstruction $f^\dag$ defined in \req{inv2}.
Therefore even in the direct reconstruction (top right image  in Figure \ref{fg:fcirc-2}) and in the
back projection reconstruction one can see some blurred boundaries in the reconstruction.
Such artifacts are expected using limited view data \req{data-part},
see \cite{LouQui00, XuYWanAmbKuc04}.

The results of interpolation based reconstruction without
oversampling ($c=1$) are quite useless. The reconstructions are
significantly improved by using a larger oversampling factor $c$.
However, even then, the results never reach the quality of the nonuniform FFT based
reconstruction.   Moreover, the numerical effort of linear interpolation
based reconstruction is proportional to the oversampling factor, which prohibits the use
of ``very large'' values for $c$ (see Figure  \ref{fg:performance}).
In the reconstruction with $c=2$ (bottom line in Figure \ref{fg:fcirc-1} and middle
line in Figure \ref{fg:fphant}) artifacts are still clearly visible.

The images in the middle line of Figure \ref{fg:fcirc-2} suggest that truncated $\sinc$ and nonuniform FFT based
reconstruction seem to perform quite similar. However, the differences to the direct Fourier reconstructions, shown in the bottom line in Figure \ref{fg:fcirc-2}, demonstrate the higher accuracy of the nonuniform FFT based algorithm.

The results in Figure \ref{fg:fcirc-3} show that all applied reconstruction algorithms are quite
stable with respect to data perturbation. In particular, the filtered back projection algorithm
produces images with the highest signal to noise ratio. However, only at the cost of a nearly 100 times longer
computation time (see Table \ref{tb:performance}).

\section{Conclusion}
We presented a novel fast Fourier reconstruction algorithm for
photoacoustic imaging using a limited planar detector array. The
proposed algorithm is based on the nonuniform FFT. Theoretical
investigation as well as numerical simulations show that our
algorithm produces better images than existing Fourier algorithms
with a similar  numerical complexity. Moreover the proposed algorithm
has been shown  to be stable against data perturbations.

\appendix[Sampling and Resolution]

Let $f$ be smooth function that vanishes outside $(0,X)^d$, and define $g$,
$f^\dag$ by \req{data-part}, \req{inv2}.
We further assume that $\Ft w_{\rm cut}$ is concentrated around zero and, that $f$ is
essentially bandlimited with essential bandwidth $\Omega$, in the
sense that $(\Ft f) (\f K)$ is negligible for $\abs{\f K}
\geq \Omega$. Note that since $f$ has bounded support, $\Ft f$
cannot vanish exactly on $\set{ \abs{\f K} \geq \Omega }$.

\begin{itemize}

  \item \textbf{Sampling of $g$.}
  Equation \req{inv} implies that
\begin{equation}\label{eq:shann-g}
    (\Ft g) (K_x, \omega)
    = \left( \Ft w_{\rm cut} \right) \ast \left( \Ft \Qo f \right) (K_x, \omega)\,,
\end{equation}
with
\begin{equation*}
    (\Ft \Qo f)
    (K_x, \omega)
    =
    \frac{2 \omega   (\Ft f) \left( K_x, \sign(\omega) \sqrt{\omega^2-\abs{K_x}^2}\right) }
        {\sign(\omega) \sqrt{\omega^2 - \abs{K_x}^2} }
   \end{equation*}
if $ \abs \omega  > \abs{K_x}^2$ and $    (\Ft \Qo f) (K_x, \omega)  = 0$ otherwise.
The assumption that $f$ has essential bandwidth $\Omega$ and
equation \req{shann-g}  imply that $(\Ft g) (K_x, \omega)$ is
negligible outside the set
\begin{equation*}
    \set{(K_x, \omega) :
    \abs{K_x} \leq \abs{\omega}  \leq \Omega }
    \subset
    (-\Omega, \Omega)^d \,.
\end{equation*}
Now Shannon's sampling theorem \cite{NatWue01,Uns00} states that  $g$
is sufficiently fine sampled if the step size in
$x$ and in $t$ satisfies the Nyquist condition $\Delta_{\step} \leq \pi / \Omega$.

\item \textbf{Sampling of $f^\dag$.}
Similar considerations as above again show that $f^\dag$ is
essentially bandlimited with essential bandwidth $\Omega$.
Shannon's sampling theorem implies that $f^\dag$ can be reliable
reconstructed from discrete samples taken with step size
$\Delta_{\step} \leq \pi / \Omega$.
\end{itemize}

If $f$ has essential bandwidth larger than $\Omega$, the function
$g$ has to be filtered with a low pass-filter \emph{before}
sampling. Otherwise,  sampling introduces aliasing  artifacts in the
reconstructed image.

Theoretically, the resolution (at least of the visible parts) can be increased ad infinity by simply
decreasing the sampling size $\Delta_{\step}$. In practical applications,
several other factors such as the bandwidth  of the ultrasound detection system
limit the bandwidth of the data, and therefore the resolution of  reconstructed images \cite{XuWan06}.
This, however, also guarantees that in practice a moderate sampling step size
$\Delta_{\step}$ gives correct sampling without aliasing.

%

\section*{Acknowledgement}
This work has been supported by the Austrian Science Fund (FWF)
within the framework of the NFN ``Photoacoustic Imaging in Biology and Medicine'',
Project S10505-N20.
Moreover, the work of M. Haltmeier has been supported by the technology transfer office
of the University Innsbruck (transIT).

\bibliographystyle{plain}
\def\cprime{$'$} \def\cprime{$'$} \def\cprime{$'$}

\end{document}